\documentclass[11pt]{article}

\usepackage[T1]{fontenc}
\usepackage{lmodern}
\usepackage[margin=1in]{geometry}

\usepackage{amsmath,amssymb,mathtools}
\usepackage{amsthm}
\usepackage{microtype}
\usepackage{needspace}

\usepackage{thmtools}
\usepackage{thm-restate}

\usepackage{xcolor}
\usepackage[hidelinks]{hyperref}
\usepackage{cleveref}

\makeatletter
\renewcommand\thmt@autorefsetup{\@xa\def\csname\thmt@envname autorefname\@xa\endcsname\@xa{\thmt@thmname}}
\makeatother

\theoremstyle{plain}
\newtheorem{theorem}{Theorem}
\newtheorem{corollary}[theorem]{Corollary}
\newtheorem{question}[theorem]{Question}
\newtheorem{lemma}[theorem]{Lemma}

\theoremstyle{remark}
\newtheorem{remark}[theorem]{Remark}

\DeclareMathOperator{\ind}{ind}
\DeclareMathOperator{\coind}{coind}
\DeclareMathOperator{\Hess}{Hess}

\DeclareMathOperator{\diag}{diag}
\newcommand{\R}{\mathbb R}
\newcommand{\Z}{\mathbb Z}
\newcommand{\RP}{\mathbb {RP}}

\newcommand{\nuplus}{\nu_{\geq 0}}

\title{
A
Topological Figiel--Lindenstrauss--Milman Theorem and its Applications}

\author{
Florian Frick\thanks{Department of Mathematical Sciences, Carnegie Mellon University. Email: \texttt{frick@cmu.edu}.}
\and
Kaave Hosseini\thanks{Department of Computer Science, University of Rochester. Email: \texttt{kaave.hosseini@rochester.edu}.}
\and
Eric Myzelev\thanks{Department of Mathematical Sciences, Carnegie Mellon University. Email: \texttt{etmyzele@andrew.cmu.edu}.}
\and
Arya Narnapatti\thanks{Department of Mathematical Sciences, Carnegie Mellon University. Email: \texttt{anarnapa@andrew.cmu.edu}.}
\and
Aliaksei Vasileuski\thanks{Department of Mathematical Sciences, Carnegie Mellon University. Email: \texttt{avasileu@andrew.cmu.edu}.}
}

\date{September 17, 2026}

\begin{document}
\maketitle
\vspace{-1.5em}
\begin{abstract}
We prove a topological generalization of
the Figiel--Lindenstrauss--Milman inequality, answering a recent
question of Frick, Hosseini, and Vasileuski: every nonempty
strongly regular CW complex with $v$ vertices,
$f$ maximal cells, and a free involution that permutes its
cells has $\Z/2$-index at most $4\log v\log f$. The classical inequality for
centrally symmetric polytopes is the special case of polytope
boundaries.

As an application, we show that every
finite simplicial complex homotopy equivalent to $\RP^d$ has at least
$\exp(\Omega(\sqrt d))$ vertices. The previous best lower bound was $\Omega(d^2)$. We also construct triangulations with
$\exp(O(\sqrt d))$ vertices.   Using the same  method, we prove a
low-dimensional Postnikov factorization theorem, which gives
similar vertex lower bounds for several classical families, such as complex and quaternionic
projective spaces, the $d$-torus, closed symplectic and aspherical manifolds,
Grassmannians, classical Lie groups, and spherical space forms. 
Further applications include results in learning theory, convex geometry, computational topology, and topological bounds for chromatic numbers and Vietoris--Rips complexes.
\end{abstract}
\tableofcontents

 \section{Introduction}

How few vertices can a triangulation of a topological space $X$ have?
Let $\Delta(X)$ denote the minimum number of vertices of a
simplicial complex homeomorphic to $X$.
This problem dates back at least to M\"obius, who in 1861
described a triangulation of the torus with seven vertices, which is 
the minimum possible~\cite{Moebius1886}.
For higher-dimensional tori $T^d$, the  previously
known general bounds were:
\[
\binom{d+2}{2}\leq\Delta(T^d)\leq 2^{d+1}-1.
\]
The lower bound follows from the cup-length estimate of
Arnoux and Marin~\cite{AM}, while the upper bound is due to
K\"uhnel and Lassmann~\cite{KuhnelLassmann1988}.
We prove the first superpolynomial lower bound:
\[
\Delta(T^d)\geq\exp(\Omega(\sqrt d)).
\]

Another classical example is that of real projective space $\RP^d$.
Here, Arnoux and Marin~\cite{AM} strengthened the same
cup-length bound to
\(
\Delta(\RP^d)\geq\binom{d+2}{2}+1
\) for $d\geq3$. 
On the upper-bound side, Adiprasito--Avvakumov--Karasev~\cite{AAK}
gave the first subexponential construction showing 
$\Delta(\RP^d)\leq\exp(O(\sqrt d\log d))$, and
Frankl--Pach--P\'alv\"olgyi~\cite{FPP} subsequently improved
this to $\Delta(\RP^d)\leq\exp(O(\sqrt{d\log d}))$.
We prove that for every $d\geq1$,
\[
\frac12\exp(\sqrt d/2)
\leq\Delta(\RP^d)
\leq\exp(10\sqrt d).
\]

We obtain similar superpolynomial lower bounds for many other
classical spaces, including complex and quaternionic projective
spaces, symplectic and aspherical manifolds, Grassmannians,
and classical Lie groups. See \Cref{cor:geometric-factorization}.
In fact, our lower bounds hold more generally for finite
simplicial complexes homotopy equivalent to the corresponding
spaces.

\paragraph{Notation.}
All logarithms are natural unless explicitly written as $\log_2$.
A \emph{free $\Z/2$-space} is a space~$X$ with an involution
$\tau$ such that $\tau(x)\ne x$ for all~$x$. Spheres carry the
antipodal involution, and a map between free $\Z/2$-spaces is
\emph{equivariant} if it commutes with the involutions. The
\emph{$\Z/2$-index} and \emph{coindex} of a free $\Z/2$-space $X$ are
\begin{align*}
 \ind X&=\min\{k\geq0:\text{there is an equivariant map }X\to S^k\},\\
 \coind X&=\sup\{k\geq0:\text{there is an equivariant map }S^k\to X\},
\end{align*}
with $\ind X=+\infty$ if no such map exists, $\coind X$ possibly
infinite, and $\ind\varnothing=\coind\varnothing=-1$. The index is
monotone under equivariant maps, and the Borsuk--Ulam theorem states
that $\ind S^d=d$; consequently $\coind X\leq\ind X$. See Matou\v sek's
book~\cite{Matousek} for an introduction.

\paragraph{}
Our lower bounds use a topological extension of the
Figiel--Lindenstrauss--Milman inequality~(FLM). The classical inequality
states that a centrally symmetric $d$-dimensional polytope with
$v$ vertices and $f$ facets satisfies
\begin{equation}\label{eq:flm}
 d\leq C \log v\,\log f
\end{equation}
for an absolute constant $C>0$ \cite{FLM}.
A topological extension of the FLM inequality to CW complexes was suggested by Frick, Hosseini, and
Vasileuski~\cite[Question~37]{FHV} while studying the power and
limitations of sign-rank lower bounds obtained from the
\(\mathbb Z/2\)-index of the so-called sign complex of a matrix.
A CW complex is \emph{strongly regular} if its characteristic
maps are embeddings and the intersection of two closed cells is
empty or a closed cell. Its \emph{facets} are its inclusion-maximal
cells.
\begin{question}[{\cite[Question~37]{FHV}}]\label{question:z2upperbound}
Is it true that any finite strongly regular 
free $\mathbb Z/2$-CW-complex $K$ with $v$ vertices and $f$ facets, satisfies 
\(
\ind(K)\le O(\,\log v\,\log f) ?
\)
\end{question}
We answer this question affirmatively.

\begin{theorem}[Main lemma: topological FLM inequality]\label{thm:main}
Let $X$ be a nonempty finite strongly regular CW complex with a free
involution that permutes its cells. If $X$ has $v$ vertices and $f$
facets, then
\begin{equation}\label{eq:main}
 \ind X\leq 4\log v\,\log f.
\end{equation}
In particular, this holds for every finite simplicial complex with a
free simplicial involution.
\end{theorem}
\begin{remark}
It is easy to see that this generalizes the classical FLM inequality. Indeed, if $X=\partial P$ for a centrally symmetric
$d$-polytope, radial projection gives an equivariant homeomorphism
$X\cong S^{d-1}$. Hence $\ind X=d-1$, and \eqref{eq:main} gives
$d\leq 4\log v\,\log f+1$. The factor of 4 matches the best known bound for the classical FLM inequality~\cite{AubrunMO}.   More discussion and history appears in Section \ref{sec:history}.
\end{remark}
\paragraph{Idea of proof of Theorem \ref{thm:main}.}
First, we reduce to
the simplicial case without increasing the numbers of vertices
or facets. Let $K$ be a simplicial complex with free $\Z/2$-action, $v$ vertices, and $f$ facets. The starting point is the sign matrix representation of Frick, Hosseini,
and Vasileuski~\cite[Lemma~15]{FHV}: encode
$K$ by a signed incidence matrix $A$, with rows and columns
indexed by representatives from antipodal pairs of facets and vertices, respectively. Each entry is picked from $\{1,-1,0\}$ and records whether a chosen facet contains a chosen vertex,
its antipode, or neither. 
Realizing antipodal vertices as $\pm e_j$ identifies the geometric realization
\[
 |K|=\{x:\|x\|_1=1, \|Ax\|_\infty=1\}.
\]
Using this description, we embed $|K|$ equivariantly into a smooth
manifold with an involution. On this manifold we construct an
invariant $C^2$ function $\mathcal F$ with a compact superlevel
set containing the embedded copy of $|K|$ but no fixed points
of the involution. We bound the index of this superlevel set
using Morse theory.
We design \(\mathcal F\) so that its Hessian has only \(4\log v\log f\) nonnegative eigenvalues. Morse theory then implies that the relevant superlevel set has the equivariant homotopy type of a free CW complex of dimension at most \(4\log v\log f\), which gives the desired index bound.

The use of critical-point estimates to control topology has
a precedent in Gromov's work on nonlinear spectra and
widths~\cite[Section~1.3.A]{Gromov1988}. For ratios of $L^p$
and $L^q$ norms on a finite set, he computes the critical
points and their Morse indices to determine the topological
size of the sublevel sets. Here, the function is instead
built from signed vertex--facet incidence data.

 \subsection{Triangulations of real projective spaces}
 By a triangulation of a space we will always mean a simplicial complex homeomorphic to it.
 To apply Theorem~\ref{thm:main} to a triangulation of $\RP^d$
with $N$ vertices, we lift it to an antipodal triangulation of
$S^d$ and enlarge the lift to a free simplicial complex with
$2N$ vertices and at most $2N$ facets. This gives
$d\leq 4\log^2(2N)$ and hence the following bound; the
construction appears in Section~\ref{subsection:triangulations}.
\begin{corollary}[Projective-space triangulations: lower bound]\label{cor:rp}
Every triangulation of $\RP^d$, $d\geq1$, has at least
\begin{equation}\label{eq:rp}
 N\geq \frac12\exp\!\left(\frac{\sqrt d}{2}\right)
\end{equation}
vertices.
\end{corollary}
Frankl--Pach--P\'alv\"olgyi \cite[Theorem~1]{FPP} already proved an
$\exp(\Omega(\sqrt d))$ lower bound for the sizes of the exchange
set systems underlying the construction of~\cite{AAK}.
Corollary~\ref{cor:rp} in particular also answers \cite[Question~3.8]{AAK}, concerning centrally symmetric
polytopes whose facets through $v$ are disjoint from those through $-v$
for every vertex $v$. Such a polytope admits an equivariant boundary
triangulation whose antipodal quotient is simplicial, without adding
vertices \cite[Corollary~2.2]{AAK}. Its number of vertices is therefore
also $\exp(\Omega(\sqrt d))$ in dimension $d$.

More generally, the same vertex bound holds for any finite simplicial
complex $T$ admitting a double cover of Stiefel--Whitney height at
least~$d$, that is, there is a cohomology class $\alpha\in H^1(T;\mathbb F_2)$ with
$\alpha^d\ne0$; see Corollary~\ref{cor:cup}.

 We complement the lower bound in Corollary~\ref{cor:rp}  with a new construction of triangulations of real projective space with $\exp(O(\sqrt d))$ vertices, improving the previous upper bound $\exp(O(\sqrt{d\log d}))$ of \cite{FPP}. The construction, given in Section~\ref{sec:rp-upper}, starts from small set systems satisfying the weak exchange condition of~\cite{FPP}. We realize such a family as the vertex set of a centrally symmetric polytope whose antipodal quotient is a simplicial complex, and construct the required families from recursively defined cographs. 
\begin{restatable}[Projective-space triangulations: upper bound]{theorem}{thmrpupper}\label{thm:rp-upper}
For every $d\geq1$, real projective space $\mathbb{RP}^d$ has a triangulation with at most $2^{10\sqrt{d+1}}$ vertices.
\end{restatable}

\begin{remark}
The same bounds apply to minimum size of a finite good cover of \(\mathbb{RP}^d\). Indeed, by the nerve lemma, the nerve of any finite good cover of \(\mathbb{RP}^d\) is a simplicial complex homotopy equivalent to \(\mathbb{RP}^d\), and hence has at least
\(
\frac12\exp\!\left(\frac{\sqrt d}{2}\right)
\)
vertices. Conversely, the open vertex stars of the triangulations from Theorem~\ref{thm:rp-upper} form good covers. Thus the minimum size of a finite good cover of \(\mathbb{RP}^d\) is
\(
\exp(\Theta(\sqrt d)).
\)
\end{remark}
\subsection{Other applications}
\paragraph{Further vertex bounds.}

The same Morse method gives a factorization theorem for maps
into spaces with homotopy groups concentrated in bounded degrees.
A CW complex is an \emph{$r$-type} if its homotopy groups vanish
above degree $r$. Write $\eta_r\colon X\to P_rX$ for the Postnikov map,
which retains the components and homotopy groups through degree $r$. For details, see~\cite[Sections~4.3]{Hatcher}.

\begin{restatable}[Low-dimensional Postnikov factorization]{theorem}{thmpostnikovcompression}
\label{thm:postnikov-compression}
\label{thm:cohomological-compression}\label{thm:classifying-compression}
For each integer $r\geq1$ there is a constant $C_r=2^{O(r^2)}$
such that every finite simplicial complex $K$ on $N\geq2$ vertices
admits maps
\begin{equation}\label{eq:postnikov-factorization}
 |K|\xrightarrow{j}Z\xrightarrow{v}P_r|K|,
 \qquad vj\simeq\eta_r,\qquad \dim Z\leq C_r\log^2N,
\end{equation}
where $Z$ is a finite CW complex. The same $j$ has the following
property: every map $f\colon |K|\to T$ to any CW $r$-type factors up to
homotopy as $f\simeq\bar f j$ for some $\bar f\colon Z\to T$.
If $K$ is connected, $Z$ may be chosen connected.
\end{restatable}
 This gives superpolynomial vertex lower bounds for many classical families, summarised below. The minimum number of vertices of a simplicial complex
homotopy equivalent to $X$ is called the \emph{covering type}
of $X$, denoted $\operatorname{ct}(X)$ which lower bounds $\Delta(X)$.

\begin{restatable}[Triangulation bounds]
{corollary}{corgeometric}
\label{cor:geometric-factorization}
Let $X$ be one of the following spaces, with $D$
defined accordingly. Then
\begin{equation}\label{eq:geometric-factorization}
\operatorname{ct}(X)\geq
\exp\!\left(\Omega\left(\sqrt D\right)\right).
\end{equation}
\begin{enumerate}
\item $X$ is a real or quaternionic projective space,
a real Grassmannian, a closed connected symplectic
or aspherical manifold, or a spherical quotient
$S^d/\Gamma$, where $\Gamma$ is a nontrivial finite
group acting freely on $S^d$ and the quotient has
finite simplicial homotopy type. Take
$D=\dim_{\mathbb R}X$.

\item $X$ is $\operatorname{SO}(n)$,
$\operatorname{U}(n)$, or $\operatorname{SU}(n)$,
where $n\geq 2$. Take $D=n$.

\item $X=\prod_{i=1}^s S^{d_i}$, where
$1\leq d_i\leq r$. Take
$D=(\sum_{i=1}^s d_i)/C_r$, where $C_r=2^{O(r^2)}$.
\end{enumerate}
\end{restatable}

 The proof appears in \Cref{app:geometric-applications}.

 \paragraph{Computational topology.}
Our bounds also give quasipolynomial algorithms for cohomological
invariants of simplicial complexes presented by lists of their
maximal simplices. Given the list of maximal simplices of a simplicial complex together with a free
simplicial involution, the Stiefel--Whitney height of the action
can be computed exactly in time $\exp(O((\log L)^3))$, where $L$
is the total input bit length. For every fixed $r\geq1$ and fixed
finite field $\mathbb F$, one can also compute the subalgebra of
$H^*(K;\mathbb F)$ generated by classes of degrees $1,\ldots,r$,
including a graded basis, its multiplication table, and its
cup-length, in time $\exp(O_r((\log L)^3))$.
In particular, this computes the entire cohomology ring and its
cup-length whenever the ring is generated in these degrees. See Theorem~\ref{thm:computational-topology}.

 The classical cochain algorithm for Stiefel--Whitney height runs
in time polynomial in the total number of simplices
\cite[pp.~104--106]{Matousek}; polynomial-time algorithms for
cohomology rings and cup-length are also known with this measure
of input size~\cite{GonzalezDiazReal2002,DeyRathod2024}.
The total number of simplices can be exponential in the number
of vertices even when the number of maximal simplices is bounded.
\paragraph{Sign complexes and shattering parameters.}
As an approach to develop lower bounds on sign-rank of sign matrices,
Frick, Hosseini, and Vasileuski~\cite{FHV} associated
to every partial sign matrix
$A\in\{-1,+1,*\}^{M\times N}$ a free
$\mathbb Z/2$-simplicial complex $S(A)$ called the sign complex of $A$.
Its vertices are signed column labels $j^+,j^-$,
and its simplices are the subsets of signed rows
of $A$ and their antipodes.

We bound the index of $S(A)$ in terms of two
shattering parameters. Let $d_2(A)$ be the largest
number of columns on which every sign pattern
involving at most two coordinates is realized
by a row of $A$ or its negative, and let
$d_\infty(A)=\operatorname{VC}(A)$ be the usual
VC dimension, namely, the largest number of columns on which every $\{-1,1\}$-sign pattern is realized.
For $N\geq 2$, we prove
\begin{equation}\label{eq:d2}
    \ind|S(A)|\leq 4d_2(A)\log^2(2N)
\end{equation}

and
\begin{equation}\label{eq:dinfty}
    \ind|S(A)|
\leq C\min\bigl\{
d_\infty(A)^2\log^2 N,\,
d_\infty(A)\log^3 N
\bigr\}.
\end{equation}
For total sign matrices, the latter improves to
$O(d_\infty(A)\log^2 N)$.
See \Cref{app:sign-complexes}.
 \paragraph{Topological bounds for chromatic numbers.}
For a flag complex $K$ with free $\Z/2$-action, $d_2(K)$
is the largest number of antipodal vertex pairs inducing
a crosspolytope boundary. Hence, if $K$ has
$v$ vertices,
\[
\ind|K|\leq 4d_2(K)\log^2 v.
\]
For a graph $G$, its box complex $B_0(G)$ is the free
$\mathbb Z/2$-complex encoding complete bipartite subgraphs of~$G$,
with the involution exchanging the two sides. In this case
$d_2(B_0(G))=\omega(G)$, the clique number of $G$, and therefore
\[
\omega(G)-1\leq \coind|B_0(G)|\leq \ind|B_0(G)|
\leq 4\omega(G)\log^2(2n),
\]
where $n$ is the number of vertices of $G$. 
Consequently, the standard chromatic lower bounds obtained from
the connectivity, Stiefel--Whitney height, coindex, or index of
the box complexes are all at most
$\ind|B_0(G)|+2=O(\omega(G)\log^2(2n))$;
see~\cite{Matousek,Csorba2007}.

It is known that for any $\varepsilon\in (0,1)$, it is NP-hard to approximate the clique number of a graph on $n$ vertices with approximation factor of $n^{1-\varepsilon}$; see~\cite{Zuckerman2007}. Combining with the inequality above immediately implies the same hardness  of 
approximation for $\coind|B_0(G)|$ and $ \ind|B_0(G)|$ given input graph $G$.  This answers 
 questions on approximation of these parameters highlighted
in~\cite{DaneshpajouhMeunier2025}. See \Cref{sec:graphs}.
\paragraph{Convex geometry and normed spaces.}
The two inequalities \ref{eq:d2} and \ref{eq:dinfty} also have natural geometric interpretations.
First, the $d_2$-bound implies that every $d$-dimensional polytope
$P$ with $n$ vertices contains at least
$
\frac{d}{4\log^2(2n)}
$
pairwise antipodal vertices, meaning that every pair lies on distinct
parallel supporting hyperplanes. The corresponding translates $P-p$
have pairwise disjoint interiors and share the origin as a vertex.
No symmetry or simpliciality assumption on $P$ is required.

For a $d$-dimensional normed space whose unit ball has at most $2N$
vertices, the same $d_2$-bound yields a centrally symmetric
equilateral set of vertices of common distance $2$ and size at least
\(
\frac{d-1}{2\log^2(2N)}.
\)
The $d_\infty$-bound implies that the space contains a
linearly isometric copy of $\ell_1^s$ with
\[
s\geq c\max\left\{
\frac{\sqrt{d-1}}{\log(2N)},\,
\frac{d-1}{\log^3(2N)}
\right\},
\]
where $c>0$ is absolute. In particular, every $d$-dimensional quotient
of $\ell_1^N$ with $N\leq d^A$, for fixed $A\geq1$, contains an
isometric $\ell_1^s$ with
$
s=\Omega_A\!\left(\frac{d}{\log^3 d}\right).
$
On the other hand, the largest dimension of a quotient of
$\ell_1^N$ containing no isometric $\ell_1^2$ is
$\Theta(\log^2 N)$. See Corollary~\ref{cor:antipodal-polytope} and
Subsection~\ref{app:banach-spaces}.

 \paragraph{Vietoris--Rips complexes of spheres.}
Vietoris--Rips complexes are flag complexes, so the $d_2$-bound (equation \ref{eq:d2}) applies.
For a finite antipodal set $V=\{\pm v_1,\ldots,\pm v_N\}\subset S^d$
and $\pi/2\leq r<\pi$, the parameter $d_2$ is exactly the largest
number of projective points $[v_i]\in\mathbb{RP}^d$ that are pairwise
separated by at least $\pi-r$. Hence
\[
\ind|\operatorname{VR}(V,r)|
\leq 4P_d(\pi-r)\log^2(2N),
\]
where $P_d(\delta)$ denotes the largest size of a
$\delta$-separated subset of $\mathbb{RP}^d$. In particular, when
\(
\frac{\pi}{2}\leq r\leq
\frac{\pi}{2}+\frac{1}{2\sqrt{d+1}},
\)
a Gram-matrix bound gives $P_d(\pi-r)=O(d+1)$, and therefore
\[
\ind|\operatorname{VR}(V,r)|
=O\!\bigl((d+1)\log^2|V|\bigr).
\]

The same argument can be discretized to bound the full
complex $\operatorname{VR}(S^d,r)$. If
$0<2\varepsilon<\pi-r$, then
\[
\ind|\operatorname{VR}(S^d,r)|
\leq
4P_d(\pi-r-2\varepsilon)
\log^2\!\bigl(2C_d(\varepsilon)\bigr),
\]
where $C_d(\varepsilon)$ is the corresponding projective covering
number, namely, the minimum number of projective
$\varepsilon$-balls needed to cover $\mathbb{RP}^d$. For fixed $\pi/2<r<\pi$, taking $\varepsilon=d^{-2}$ makes
the covering factor subexponential. Using bounds for projective codes,
this improves the direct covering-number bound of
Adams--Bush--Virk~\cite{ABV}; for example, at $r=2\pi/3$ the latter
gives approximately $2^d$, while the packing bound gives
approximately $1.321^d$.
See \Cref{sec:rips}.
 
\subsection{Historical context}\label{sec:history}

\paragraph{Triangulations with few vertices.}
The study of triangulations with few vertices dates back at least
to M\"obius, who between 1861 and 1865 described six-, seven-, and
eight-vertex triangulations of the real projective plane, the torus,
and the Klein bottle, respectively.
All three constructions appear in his posthumously published
collected works~\cite[pp.~482--483, 521, 552--553]{Moebius1886}. Cairns~\cite{Cairns1935} established the existence of triangulations
for smooth manifolds, and Brehm and K\"uhnel~\cite{BrehmKuhnel1987}
proved that every combinatorial triangulation of a closed
$d$-manifold not homeomorphic to $S^d$ has at least $3d/2+3$
vertices.

\medskip
\noindent\emph{Projective spaces.}
In low dimensions, small triangulations are known for
$\mathbb{RP}^{2},\mathbb{RP}^{3},\ldots,\mathbb{RP}^{6}$; see
\cite{Alexandroff1932,Lutz2005,Balagopalan2017,
VenturelloZheng2021,GuyerSteinerbergerYang2026}.
Small triangulations are also known for
$\mathbb{CP}^{2}$, $\mathbb{CP}^{3}$, and $\mathbb{HP}^{2}$; see
\cite{KuhnelBanchoff1983,BagchiDatta2012,BrehmKuhnel1992,
Gorodkov2019,Gaifullin2024}.

In growing dimension, K\"uhnel~\cite{Kuhnel1987} constructed
triangulations of $\mathbb{RP}^{d}$ with $2^{d+1}-1$ vertices,
and Venturello and Zheng~\cite{VenturelloZheng2021} improved this
to $O(\varphi^d)$, where $\varphi=(1+\sqrt5)/2$.
Adiprasito, Avvakumov, and Karasev~\cite{AAK} obtained the first
subexponential construction, with
$\exp(O(\sqrt d\log d))$ vertices, and Frankl, Pach, and
P\'alv\"olgyi~\cite{FPP} improved this to
$\exp(O(\sqrt{d\log d}))$.
For complex projective space, Datta and
Spreer~\cite{DattaSpreer2026} constructed triangulations of
$\mathbb{CP}^{m}$ with $\exp(O(m\log m))$ vertices.
The known general lower bounds, however, were only quadratic.
Arnoux and Marin~\cite{AM} proved
\[
  \Delta(\mathbb{RP}^{d})\geq\binom{d+2}{2}+1
  \quad(d\geq3),
  \qquad
  \Delta(\mathbb{CP}^{m})\geq(m+1)^2,
\]
with equality in the second bound possible only for $m\leq2$.
Frankl, Pach, and P\'alv\"olgyi~\cite{FPP} also proved an
$\exp(\Omega(\sqrt d))$ lower bound for the exchange set systems
used in the Adiprasito–Avvakumov–Karasev approach of $\mathbb{RP}^{d}$, but not
for arbitrary triangulations.
 
\medskip
\noindent\emph{Other spaces.}
There are also explicit constructions for other families.
K\"uhnel and Lassmann~\cite{KuhnelLassmann1988} constructed
$d$-dimensional tori with $2^{d+1}-1$ vertices, while Klee and
Novik~\cite{KleeNovik2012} constructed centrally symmetric
triangulations of $S^p\times S^q$ with $2(p+q+2)$ vertices.

A general source of vertex lower bounds is the covering type
introduced by Karoubi and Weibel~\cite{KW}, which satisfies
$\operatorname{ct}(M)\leq\Delta(M)$.
Govc, Marzantowicz, and Pave\v{s}i\'c~\cite{GMP}
developed estimates from the cohomology ring and
Lusternik--Schnirelmann category. These give quadratic vertex
bounds for tori, quaternionic projective spaces, and products of an
increasing number of spheres of uniformly bounded dimension, and
cubic bounds in $n$ for $\operatorname{U}(n)$ and
$\operatorname{SU}(n)$. Their methods also give quadratic vertex
bounds in the ambient dimension for real
Grassmannians~\cite{GMPGrassmannians2020}.
Duan, Marzantowicz, and Zhao~\cite{DMZ} obtained cubic vertex
bounds in $n$ for $\operatorname{SO}(n)$ and the quadratic bound
$\operatorname{ct}(M)\geq(m+1)^2$ for closed symplectic
$2m$-manifolds. Category-weight methods give further polynomial
bounds for lens spaces and related quotients
\cite{GMPCategoryWeight}.
For closed aspherical $d$-manifolds, the category theorem of
Eilenberg and Ganea~\cite{EilenbergGanea1957}, together with
covering-type estimates~\cite{GMP}, gives a quadratic vertex bound.
Pave\v{s}i\'c~\cite{Pavesic2019} obtained further estimates for
PL triangulations from the fundamental group, and Avvakumov,
Balitskiy, Hubard, and Karasev~\cite{ABHK2024} proved quadratic
vertex bounds for topologically essential complexes by
systolic methods.
Thus,   all previously known  
\emph{vertex} lower bounds for the mentioned spaces were polynomial in the relevant
dimension or rank parameter.

\paragraph{Figiel--Lindenstrauss--Milman inequality.}
The Figiel--Lindenstrauss--Milman (FLM) inequality, proved in 1977~\cite{FLM},
states that every centrally symmetric \(d\)-dimensional polytope with
\(v\) vertices and \(f\) facets satisfies
\(
    \log v\,\log f\geq c d
\)
for an absolute constant \(c>0\). In particular, the numbers of vertices
and facets cannot both grow polynomially with the dimension. Although
its statement is purely combinatorial, the inequality arose from the
study of almost Euclidean sections of convex bodies, and its original
proof used Milman's version of Dvoretzky's theorem.

The dependence on dimension is sharp up to an absolute factor.
Indeed, the original work already gave examples of Hanner polytopes
with \(\log v,\log f=\Theta(\sqrt d)\); see also
\cite[Proposition~4.1]{Milo}. Stronger bounds are possible in certain
ranges: Barvinok~\cite{Barvinok2013} proved that a centrally symmetric
\(d\)-polytope with \(O(d)\) facets must have \(\exp(\Omega(d))\)
vertices, whereas FLM alone gives only
\(\exp(\Omega(d/\log d))\).

In 2025, Milo~\cite{Milo} gave a short elementary proof avoiding
Dvoretzky's theorem, with asymptotic constant \(c=1/9+o(1)\).
Aubrun observed in a MathOverflow discussion~\cite{AubrunMO} that a
Gaussian variant of Milo's argument gives \(c=1/4+O(1/d)\). 
The optimal asymptotic constant is, to our knowledge, unknown.

\paragraph{Organization of the paper.}
Section~\ref{sec:flm} proves the matrix bound and the topological
FLM inequality. Section~\ref{sec:triangulations} establishes the
matching lower and upper bounds for triangulations of real
projective spaces. 
\Cref{app:compression-applications}  proves the Postnikov factorization theorem, and
derives its geometric consequences.
\Cref{app:sign-complexes} combines
the flag-complex and VC-dimension estimates in one shattering
theorem and develops their graph-theoretic and normed-space
applications. 

\section{The topological FLM inequality}\label{sec:flm}

 \subsection{Preliminary lemmas}
Let $\Sigma$ be a smooth manifold without boundary and let $f\colon\Sigma\to\mathbb{R}$ be smooth. For $x\in\Sigma$ and $u\in T_x\Sigma$, we write $d_\Sigma f(x)[u]$ for the directional derivative of $f$ at $x$ in the tangent direction $u$. Thus $d_\Sigma f(x)=0$ means that all directional derivatives vanish at $x$. At such a critical point, $\Hess_\Sigma f(x)$ denotes the quadratic form on $T_x\Sigma$ given by the second derivative of $f$ at $x$. For a quadratic form $H$, let $\nuplus(H)$ denote the number of nonnegative eigenvalues of $H$, counted with multiplicity. Equivalently, $\nuplus(H)$ is the largest dimension of a subspace on which $H$ is nonnegative. This quantity is independent of the choice of inner product. We write $H_1\preceq H_2$ when $H_2-H_1$ is positive semidefinite; then $\nuplus(H_1)\le \nuplus(H_2)$. Note also that $\nuplus$ also cannot increase when a form is restricted to a subspace.

We use the following standard consequence of Morse theory
\cite[Theorem~3.5 and the subsequent remark]{Milnor}.

\begin{theorem}[Morse cell-attachment theorem]\label{thm:morse-cells}
Let $Q$ be a compact smooth manifold with boundary, and let
$g\colon Q\to\mathbb R$ be a Morse function having no critical points
near $\partial Q$. Suppose that $g$ is constant on $\partial Q$ and
strictly smaller in the interior. 
Then $Q$ has the homotopy type of a finite CW complex
with one cell for each critical point $y$ of~$g$,
of dimension equal to the number of negative eigenvalues
of $\Hess_Q g(y)$, counted with multiplicity.
\end{theorem}

We state the following lemma for arbitrary finite groups,
although only the case $G=\mathbb Z/2$ is needed for
Theorem~\ref{thm:main}.
\begin{lemma}[Morse bound for superlevel sets]
\label{lem:morse}\label{lem:morse-finite-group}
Let a finite group $G$ act smoothly on a smooth manifold $\Sigma$
without boundary, and let $f\colon\Sigma\to\R$ be smooth and $G$-invariant.
Suppose $a<b$, the set $\{f\geq a\}$ is compact, and the action is
free on this set. Let $D\geq0$ be an integer such that
\[
 \nuplus(\Hess_\Sigma f(x))\leq D
 \qquad
 \text{whenever }d_\Sigma f(x)=0\text{ and }f(x)\geq a.
\]
Then there is an invariant subspace $W$ with
\(
 \{f\geq b\}\subseteq W\subseteq\{f\geq a\}
\)
that is $G$-equivariantly homotopy equivalent to a finite free
$G$-CW complex of dimension at most $D$.

In particular, $\{f\geq b\}$ admits an equivariant map to such a
complex. For $G=\mathbb{Z}/2$, this gives
\[
 \ind\{f\geq b\}\leq D.
\]
\end{lemma}

\begin{proof}
Choose a regular value $t\in(a,b)$ (that is, $d_\Sigma f(x)\neq 0$ whenever $f(x)=t$) and set $W=\{f\ge t\}$.
Then $\{f\ge b\}\subseteq W\subseteq\{f\ge a\}$.  If $W$ is empty,
there is nothing to prove. Otherwise, $W$ is a compact smooth
manifold with boundary carrying a free $G$-action.

Let $Q=W/G$. Since $G$ is finite and acts freely on $W$, the quotient $Q$ is again a compact smooth manifold with boundary.
The function $f|_W$ descends to a smooth function $\bar f\colon Q\to\mathbb{R}$, defined by $\bar f([x])=f(x)$. This is well defined because $f$ is $G$-invariant. Since $\partial W=\{f=t\}$, we have $\bar f=t$ on $\partial Q$, while $\bar f>t$ on the interior of $Q$.
Moreover, since $t$ is a regular value, there is an open neighborhood $U$ of $\partial Q$ in $Q$ on which $\bar f$ has no critical points.
Choose a sufficiently small $C^2$ perturbation $h$ of $\bar f$ which is Morse, agrees with $\bar f$ on the neighborhood $U$ of $\partial Q$, and still satisfies $h>t$ on the interior of $Q$ \cite[Corollary~6.8]{Milnor}. We claim that every critical point $y$ of $h$ satisfies
\(
\nuplus(\Hess h(y))\le D.
\)
To see this, let $n=\dim Q$. At every critical point $x$ of $\bar f$, the hypothesis gives $\nuplus(\Hess \bar f(x))\le D$. Equivalently, $\Hess \bar f(x)$ has at least $n-D$ strictly negative eigenvalues. Since the critical set of $\bar f$ is compact, these negative directions persist on a small neighborhood of the critical set under a sufficiently small $C^2$ perturbation. Thus, at every critical point of $h$ lying in this neighborhood, $\Hess h$ still has at least $n-D$ negative eigenvalues.
Outside this neighborhood, $d\bar f$ is bounded away from zero. Therefore, if $h$ is sufficiently close to $\bar f$ in $C^1$, then $dh$ is also nonzero there, so $h$ has no critical points outside the neighborhood \cite[Lemma~22.4]{Milnor}. Hence every critical point $y$ of $h$ satisfies $\nuplus(\Hess h(y))\le D$.

Now  \Cref{thm:morse-cells}  applied to $-h$ gives a homotopy equivalence
$Z_0\to Q$, where $Z_0$ is a finite CW complex of dimension
at most $D$. 
Pull back the covering $W\to Q$ along this equivalence.
The resulting space $Z$ is a finite free $G$-CW complex
of the same dimension, and the induced map $Z\to W$ is a
$G$-equivariant homotopy equivalence
\cite[Proposition~1.30 and Appendix, Exercise~1]{Hatcher}.

For $G=\Z/2$, choose an equivariant map from the vertices of $Z$
to $S^D$ and extend it one orbit of cells at a time.
The extensions exist because $S^D$ is $(D-1)$-connected.
Composing with $\{f\geq b\}\hookrightarrow W\to Z$ proves
the index bound.
\end{proof}

\begin{remark}
The lemma also holds for invariant $C^2$ functions.
First take a sufficiently close smooth approximation and
average it over $G$. Apply the smooth case with levels
$a',b'$ satisfying $a<a'<b'<b$.
For a sufficiently close approximation, the inclusions
between superlevel sets and the Hessian bound used above
continue to hold.

\end{remark}

\subsection{The matrix bound}

Now we prove the matrix estimate used in
the proof of \Cref{thm:main} in Section~\ref{sec:complexes}. 

\begin{theorem}[Matrix bound]\label{thm:matrix}
For $m,n\geq1$, $A\in[-1,1]^{m\times n}$, and $0<\rho\leq1$, set
\[
R_\rho(A)
=
\{x\in\R^n:\|x\|_1=1,\ \|Ax\|_\infty\geq\rho\}.
\]
With the antipodal involution,
\begin{equation}\label{eq:matrix}
\ind R_\rho(A)
\leq
\min\{m-1,n-1,4\rho^{-2}\log(2m)\log(2n)\}.
\end{equation}
 \end{theorem}

\begin{proof}
The odd maps $x\mapsto x/\|x\|_2$ and $x\mapsto Ax/\|Ax\|_2$ give the bounds $n-1$ and $m-1$. Assume $R_\rho(A)\ne\varnothing$. Put $N=2n$, $M=2m$, and
\[
B=\begin{pmatrix}A&-A\\-A&A\end{pmatrix},
\qquad
a=\frac{\rho}{2\log N},\qquad b=\frac{\rho}{2\log M}.
\]
Let $\Delta=\{p\in\R^N:p_j\ge0,\ \sum_jp_j=1\}$, with the involution swapping the two column blocks; let $\Delta^\circ$ be its relative interior. For $x\in R_\rho(A)$, write $x=x^+-x^-$ and set $\varphi(x)=(x^+,x^-)$. This map is equivariant, takes values in~$\Delta$, and satisfies $\max_i(B\varphi(x))_i=\|Ax\|_\infty\ge\rho$. We construct an invariant $C^2$ function $\mathcal F$ and a level $c$ such that $\{\mathcal F\ge c\}$ is compact, receives an equivariant map from $R_\rho(A)$, carries a free involution, and has index bounded using Morse theory.

\paragraph{Construction of $\mathcal F$.}
The fixed points of the involution on $\Delta$ are precisely the points $p=(u,u)$. For such points $Bp=0$. Thus we want to separate the image of $R_\rho(A)$ from the fixed-point set, while ensuring that $\Hess_{\Delta^\circ}\mathcal F$ has few nonnegative directions at its critical points.
The function $p\mapsto\max_i(Bp)_i$ is at least $\rho$ on the image of $R_\rho(A)$ and is $0$ on the fixed-point set. Since it is not smooth, we replace it by the soft maximum $\mathcal F'(p)=b\log\sum_i e^{(Bp)_i/b}$, with $b$ chosen above. 

Its Hessian is positive semidefinite, so this alone gives no
useful bound on the number of nonnegative directions.
We therefore seek a concave correction $\mathcal G$,
nonnegative on $\Delta$, and take
$F_0=\mathcal F'+\mathcal G$.
Ideally, $\Hess\mathcal G$ would be a negative multiple of
the identity, uniformly offsetting the positive curvature
of $\mathcal F'$. It turns out that the entropy function
achieves this after a rescaling.
Indeed, the function
\[
 \mathcal G(p)=-a\sum_{j=1}^N p_j\log p_j
\]
with the convention $0\log0=0$, is nonnegative on $\Delta$,  and strictly positive on $\Delta^\circ$. At points
with positive coordinates, its Hessian is
\(
 \Hess\mathcal G(p)=-a\diag(1/p).
\)
After the substitution $z=\diag(\sqrt p)u$ in tangent
directions, its quadratic form becomes
\[
 z^{\mathsf T}\Hess\mathcal G(p)z=-a\|u\|_2^2.
\]
  This leads to the candidate function
\[
 F_0(p)=b\log\sum_{i=1}^M e^{(Bp)_i/b}+aH(p),
 \qquad H(p)=-\sum_{j=1}^Np_j\log p_j.
\]
However there is still a compactness issue: $F_0$ is naturally defined on $\Delta$, whose boundary is not a smooth manifold. We therefore move the image of $R_\rho(A)$ slightly into $\Delta^\circ$ and add an arbitrarily small logarithmic barrier $\eta\sum_j\log p_j$. The barrier tends to $-\infty$ at the boundary and contributes only an additional negative term to the Hessian, so it gives compact superlevel sets without weakening the curvature estimate.

\paragraph{Separation and the barrier.}
Since $Bp=0$ at a fixed point $p$, $F_0(p) \leq b\log M+a\log N=\rho$.
Set
\[
 S\coloneqq\varphi(R_\rho(A))\subseteq\Delta.
\]
Then $S$ is compact and invariant under the block-swapping
involution, and $\max_i(Bp)_i\geq\rho$ for every $p\in S$.
On $S$, the soft maximum is strictly larger than its largest argument
because $M\geq2$, and $H\geq0$. Hence
$\delta\coloneqq\min_S F_0-\rho>0$.

Let $u=(1/N,\ldots,1/N)$. By continuity, for sufficiently small
$0<\lambda<1$, the equivariant inward map
$p\mapsto p_\lambda=(1-\lambda)p+\lambda u$ satisfies
$F_0(p_\lambda)\geq\rho+3\delta/4$ for every $p\in S$.
All its coordinates are at least $\lambda/N$. On the open simplex
$\Delta^\circ$, choose
$0<\eta\leq\delta/(4N\log(N/\lambda))$ and put
\begin{equation}\label{eq:entropy-function}
 \mathcal F(p)=F_0(p)+\eta\sum_{j=1}^N\log p_j.
\end{equation}
Then $\mathcal F(p_\lambda)\geq\rho+\delta/2$ for $p\in S$.
The barrier tends to $-\infty$ at the boundary, so all superlevel sets
are compact in $\Delta^\circ$. It is nonpositive, so fixed points
still have value at most $\rho$. Every superlevel above $\rho$ is
therefore free. 

\paragraph{The Hessian estimate.}
For $p\in\Delta^\circ$, set
\[
 q_i=\frac{e^{(Bp)_i/b}}{\sum_\ell e^{(Bp)_\ell/b}},\qquad
 C=\diag(\sqrt q)B\diag(\sqrt p).
\]
The Hessian on the affine tangent space of the simplex is the
restriction of
\[
 b^{-1}B^{\mathsf T}(\diag(q)-qq^{\mathsf T})B
       -a\diag(1/p)-\eta\diag(1/p^2).
\]
After $z=\diag(\sqrt p)w$, its quadratic form is bounded above by
\[
 b^{-1}C^{\mathsf T}C-aI-\eta\diag(1/p)
 \preceq b^{-1}C^{\mathsf T}C-aI.
\]

By the spectral theorem, the symmetric matrix $C^{\mathsf T}C$
has an orthonormal basis of eigenvectors with eigenvalues
$\sigma_i(C)^2$. Thus
$b^{-1}C^{\mathsf T}C-aI$ has eigenvalues
$b^{-1}\sigma_i(C)^2-a$, which are nonnegative precisely when
$\sigma_i(C)^2\geq ab$.

Since the rescaled Hessian is bounded above by this matrix,
the dimension of any subspace on which the Hessian is
nonnegative cannot exceed the number of these eigenvalues.
Indeed, a larger subspace would intersect the span of the
strictly negative eigenvectors, giving a contradiction.
Consequently,
\begin{equation}\label{eq:entropy-count}
 \nuplus(\Hess_{\Delta^\circ}\mathcal F(p))
 \leq\#\{i:\sigma_i(C)^2\geq ab\}
 \leq\frac{\|C\|_F^2}{ab}
 \leq\frac1{ab},
\end{equation}
where the second inequality follows from
$\sum_i\sigma_i(C)^2=\|C\|_F^2$, and the last from
${\|C\|_F^2=\sum_{i,j}q_ip_jB_{ij}^2\leq1}$.

\paragraph{Conclusion.}
Choose $\rho<c_0<c_1<\rho+\delta/2$. The map
$x\mapsto \varphi(x)_\lambda$ lands in $\{\mathcal F\geq c_1\}$, and
$\{\mathcal F\geq c_0\}$ is compact and free. Lemma~\ref{lem:morse}
therefore gives
\(
 \ind R_\rho(A)\leq\left\lfloor\frac1{ab}\right\rfloor
       \leq4\rho^{-2}\log(2m)\log(2n).\qedhere
\)
\end{proof}

\subsection{Proof of the topological FLM inequality}\label{sec:complexes}

\begin{proof}[Proof of Theorem~\ref{thm:main}]
First suppose that $K$ is a finite simplicial complex with a free
simplicial involution $\tau$, with $v=2n$ vertices and $f=2m$ facets.
No simplex contains both a vertex and its antipode, since the midpoint
of the corresponding edge would be fixed by $\tau$. Choose representatives
$v_1,\ldots,v_n$ of the vertex orbits and $\sigma_1,\ldots,\sigma_m$
of the facet orbits, and define
\[
 A_{ij}=\begin{cases}
  1,&v_j\in\sigma_i,\\
  -1,&\tau v_j\in\sigma_i,\\
  0,&\text{otherwise}.
 \end{cases}
\]
Identify $|K|$ with its natural realization on the $\ell_1$-sphere,
sending $v_j$ to $e_j$ and $\tau v_j$ to $-e_j$.
If $x$ belongs to $\sigma_i$ or $\tau\sigma_i$, then
$(Ax)_i=1$ or $-1$, respectively. Thus $|K|\subseteq R_1(A)$.
Conversely, suppose $\|x\|_1=1$ and $\|Ax\|_\infty\geq1$.
Since $|A_{ij}|\leq1$, we have $\|Ax\|_\infty\leq\|x\|_1=1$.
For some row $i$, equality holds in
\[
 \left|\sum_j A_{ij}x_j\right|\leq\sum_j|x_j|=1.
\]
Consequently, for every $j$ with $x_j\ne0$, the entry $A_{ij}$ is
$1$ or $-1$, and all nonzero terms $A_{ij}x_j$ have the same sign.
Hence the signed support of $x$ lies in $\sigma_i$ or $\tau\sigma_i$,
so $x\in|K|$. We conclude that $|K|=R_1(A)$, and
Theorem~\ref{thm:matrix} gives
\(
 \ind|K|\leq4\log(2n)\log(2m)=4\log v\log f.
\)

Now let $X$ be a finite strongly regular CW complex with a free
involution $\tau$ that permutes its cells. We claim that every closed
cell $\sigma$ is disjoint from $\tau\sigma$. Indeed, if their
intersection were nonempty, strong regularity would make
$\sigma\cap\tau\sigma$ a closed cell. This cell is invariant under
$\tau$, and hence, being homeomorphic to a closed ball, would contain
a fixed point by Brouwer's fixed-point theorem, a contradiction.
In particular, no closed cell contains an antipodal pair of vertices.

Let $K$ be the simplicial complex generated by the full simplices on
the vertex sets of the facets of $X$. It has $v$ vertices and at most
$f$ facets. The involution acts simplicially on $K$ and is free,
since no simplex contains an antipodal pair of vertices.

We construct an equivariant map $\phi:X\to|K|$ by induction over
cell dimensions. Map each vertex of $X$ to the corresponding vertex
of $K$. Suppose $\phi$ has been defined equivariantly on the
$(k-1)$-skeleton, carrying every closed cell to the simplex on its
vertices. For a $k$-cell $\sigma$, the restriction of $\phi$ to
$\partial\sigma$ lands in the simplex on the vertices of $\sigma$:
every boundary cell has its vertices among those of $\sigma$.
Because this simplex is contractible, the boundary map extends over
$\overline\sigma$, which is a closed ball. Choose such an extension
on one cell from each $\tau$-orbit and define it on the paired cell
by equivariance. These choices are compatible on boundaries, so the
induction produces the desired map $\phi$.

By monotonicity of the index and the simplicial case,
\(
 \ind X\leq\ind|K|\leq4\log v\log f.
\)
\end{proof}

\begin{remark} 
[Necessity of strong regularity]
The antipodal CW structure on $S^d$
with two cells in each dimension is regular and has two vertices and
two facets, but its index is $d$; for $d\geq1$ its closed cells do not
satisfy the intersection condition.
\end{remark}

\begin{remark}[Information theoretic generalizations]
The same argument can be generalized to give an entropic formulation. Let $X$ be a nonempty compact space with
free involution $\tau$, and let $P\colon X\to\Delta_m$ and
$Q\colon X\to\Delta_n$ be continuous maps to probability simplices. Put
\[
 \mathcal I(P)=
 \inf_{\pi\in\Delta_m^\circ}\sup_{x\in X}
 \sum_{i=1}^m P_x(i)\log\frac{P_x(i)}{\pi_i},
\]
with $0\log0=0$, and define $\mathcal I(Q)$ analogously. By the information-radius characterization of channel capacity,
$\mathcal I(P)$ is the channel capacity of $x\mapsto P_x$.
If $A\in[-1,1]^{m\times n}$ and $\rho>0$ satisfy
\[
 P_x^{\mathsf T}AQ_x\geq\rho,
 \qquad
 P_x^{\mathsf T}AQ_{\tau x}\leq-\rho
 \qquad(x\in X),
\]
then
\[
 \ind X\leq4\rho^{-2}\mathcal I(P)\mathcal I(Q).
\]
Uniform reference distributions give
$\mathcal I(P)\leq\log m$ and $\mathcal I(Q)\leq\log n$.
For the proof, symmetrize the probability maps and reference
distributions by doubling their coordinates; this preserves
the relative-entropy bounds. The weighted soft-maximum and
negative-relative-entropy version of the same argument gives
\[
 \ind X\leq\frac{4UV}{(1-\varepsilon)^2\rho^2}
\]
for $U>\mathcal I(P)$, $V>\mathcal I(Q)$, and
$0<\varepsilon<1$. Here the factor $1-\varepsilon$ in the
entropy coefficients ensures strict separation from the
fixed-point set. Letting $U,V$ decrease to the information
radii and $\varepsilon\downarrow0$ proves the claim.
\end{remark}
\section{Projective space triangulations}\label{sec:triangulations}

\subsection{The lower bound}\label{subsection:triangulations}

\begin{proof}[Proof of Corollary~\ref{cor:rp}]
Let $T$ be a triangulation of $\RP^d$ with $N$ vertices.
Lift $T$ along the double covering $S^d\to\RP^d$ to obtain
a simplicial complex $K$ triangulating $S^d$. Each vertex
and simplex of $T$ has two disjoint lifts, so $K$ has $2N$
vertices and carries a free simplicial involution $\tau$,
corresponding to the antipodal map $x\mapsto-x$ on $S^d$.
In particular, $\ind|K|=d$.

For a vertex $w$ of $K$, let $N[w]$ denote its closed
neighborhood in the $1$-skeleton, consisting of $w$ and
all vertices adjacent to it. We claim that $N[w]$ contains
no pair of antipodal vertices $v,\tau v$.
Indeed, $v$ and $\tau v$ cannot be adjacent, since they
project to the same vertex of $T$. Moreover, if
$w\notin\{v,\tau v\}$ were adjacent to both $v$ and
$\tau v$, the edges $wv$ and $w\tau v$ would be two
distinct lifts of the same edge of $T$ sharing the
vertex $w$. This is impossible, since the two lifts
of an edge are disjoint.

For a vertex set $U$, write $2^U$ for the full simplex
on $U$, and define
\[
 L=\bigcup_{w\in V(K)}2^{N[w]}.
\]
The involution preserves $L$, and $L$ is free because
none of its simplices contains an antipodal vertex pair.
Furthermore, $K\subseteq L$, since every simplex of $K$
is contained in the closed neighborhood of each of its
vertices.

Finally, $L$ has $2N$ vertices and at most $2N$ facets,
one for each generating simplex. Applying
Theorem~\ref{thm:main}, we obtain
\[
 d=\ind|K|
 \leq\ind|L|
 \leq4\log^2(2N).
\]
Rearranging gives
\(
 N\geq\frac12\exp\left(\frac{\sqrt d}{2}\right),
\)
as claimed.
\end{proof}

\begin{remark}[Other spherical quotients]
The entropy--Morse argument extends to free prime-order
cyclic actions provided every simplex contains at most one
vertex from each group orbit. Consequently, if a finite
group $\Gamma$ of order $q\geq2$ acts freely on $S^d$,
every simplicial triangulation of $S^d/\Gamma$ with $N$
vertices satisfies
\[
 d\leq4\log^2(qN),
 \qquad
 N\geq\frac1q\exp\!\left(\frac{\sqrt d}{2}\right).
\]
Indeed, the lifted neighborhood enlargement has $qN$
vertices and at most $qN$ facets; uniqueness of edge lifts
gives the required orbit condition. Restrict to a subgroup
of prime order $p$, center the facet--vertex incidence
matrix on each vertex orbit, and rescale by $p/(p-1)$.
The resulting matrix has entries in $[-1,1]$, annihilates
invariant probability vectors, and retains the value $1$
on incidences. The same entropy--Morse argument gives an
equivariant map from $S^d$ to a finite free complex of
dimension at most $4\log^2(qN)$, and Dold's
theorem~\cite{Dold1983} gives the bound.
\end{remark}
\subsection{A matching upper bound}
\label{sec:rp-upper}

Our construction uses the weak exchange condition of
\cite[Theorem~1]{FPP}, rather than the stronger exchange
condition of~\cite{AAK}.

\thmrpupper*

\paragraph{The simplicial quotient.}
We first record the quotient criterion used in the proof of
\cite[Corollary~2.2]{AAK}.

\begin{lemma}\label{lem:rp-quotient}
Let $P\subseteq\mathbb R^n$ be a full-dimensional centrally symmetric
polytope, centered at the origin. Suppose that for every vertex $v$,
any facet containing $v$ is disjoint from every facet containing $-v$.
Then $\mathbb{RP}^{n-1}$ admits a triangulation with exactly
$|V(P)|/2$ vertices.
\end{lemma}

\begin{proof}
Order the antipodal pairs of vertices. A proper face contains no
antipodal pair, so this induces a total order on its vertices.
Pulling triangulations with these orders agree on common faces and
commute with negation. They therefore give an antipodal triangulation
$K$ of $\partial P$ with vertex set $V(P)$. The hypothesis implies that
the closed stars of $v$ and $-v$ in $K$ are disjoint.

Let $T$ have the vertex orbits as vertices and the images of simplices
of $K$ as simplices. Suppose two points of $|K|$ have the same image
in $|T|$. Their minimal supporting simplices have the same vertex-orbit
set. After replacing one of the points by its antipode if necessary,
we may assume that their supporting simplices share a vertex $v$. No other orbit can have opposite representatives $w$ and
$-w$ in those supports: then $v$ would belong to both closed stars of
$w$ and $-w$. Thus their supports and barycentric coordinates agree.
The fibers are therefore precisely antipodal pairs. The induced continuous bijection
$\partial P/\{\pm1\}\to|T|$ is a homeomorphism, since its
domain is compact and $|T|$ is Hausdorff. Thus $T$ triangulates
$\mathbb{RP}^{n-1}$ and has $|V(P)|/2$ vertices.
\end{proof}

\paragraph{Weak exchange and a polytope realization.}
A family $\mathcal F\subseteq2^{[n]}$ is \emph{atomic} if it contains
$\varnothing$ and every singleton, and \emph{downward closed} if it
contains every subset of each of its members. Say that $\mathcal F$
satisfies \emph{weak exchange} if, for all disjoint nonempty
$A,B\in\mathcal F$,
\begin{equation}\label{eq:rp-weak-exchange}
 \bigl(\exists b\in B:\ A\cup\{b\}\in\mathcal F\bigr)
 \quad\text{or}\quad
 \bigl(\exists a\in A:\ B\cup\{a\}\in\mathcal F\bigr).
\end{equation}

\begin{lemma}\label{lem:weak-exchange-polytope}
Let $\mathcal F\subseteq2^{[n]}$ be atomic, downward closed, and satisfy
weak exchange. Then $\mathbb{RP}^{n-1}$ has a triangulation
with $|\mathcal F|-1$ vertices.
\end{lemma}

\begin{proof}
For each nonempty $A\in\mathcal F$, let $\mathbf1_A$ be its
characteristic vector and set
\begin{equation}\label{eq:rp-polytope}
 p_A=\frac{\mathbf1_A}{|A|+1},
 \qquad
 P=\operatorname{conv}\{\pm p_A:
             A\in\mathcal F\setminus\{\varnothing\}\}.
\end{equation}
Since $P$ contains $\pm e_i/2$ for every $i$, it is full-dimensional
and contains the origin in its interior. All the displayed points
are vertices. Indeed, for fixed nonempty $A$, the functional
$z\mapsto\langle\mathbf1_A,z\rangle$ uniquely maximizes at $p_A$:
for every nonempty $B\in\mathcal F$,
\[
 \langle\mathbf1_A,p_B\rangle
 =\frac{|A\cap B|}{|B|+1}
 \leq\frac{|A\cap B|}{|A\cap B|+1}
 \leq\frac{|A|}{|A|+1},
\]
with equality throughout only when $B=A$, while its values at negative
vertices are nonpositive. Thus $|V(P)|=2(|\mathcal F|-1)$.

Fix a nonempty proper face $Q$ and a nonzero exposing functional $x$, and put
$h=\max_{z\in P}\langle x,z\rangle>0$. For $S\subseteq[n]$, write
$x(S)=\sum_{i\in S}x_i$. The defining inequalities for $P$ give
\begin{equation}\label{eq:rp-support-inequality}
 |x(S)|\leq h(|S|+1)\qquad(S\in\mathcal F).
\end{equation}
For $S=\varnothing$ this inequality holds trivially. We establish two
restrictions on vertices lying in~$Q$.

First, if $p_A\in Q$, then $x(A)=h(|A|+1)$. Downward closure and
\eqref{eq:rp-support-inequality} applied to $A\setminus\{i\}$ imply
$x_i\geq h$ for every $i\in A$. Similarly, if $-p_B\in Q$, then
$x_i\leq-h$ for every $i\in B$. Consequently,
\begin{equation}\label{eq:rp-opposite-signs}
 p_A,-p_B\in Q\quad\Longrightarrow\quad A\cap B=\varnothing.
\end{equation}

Second, suppose $p_A,p_B\in Q$ and $A\cap B=\varnothing$. Set
$y_i=x_i-h$ and $y(S)=\sum_{i\in S}y_i$. Then
\eqref{eq:rp-support-inequality} implies $y(S)\leq h$ for every
$S\in\mathcal F$, while $y(A)=y(B)=h$. Since $x_i\geq h$ for every $i\in A\cup B$, we have
$y_i=x_i-h\geq0$ on $A\cup B$. Hence
\[
 A_+=\{i\in A:y_i>0\},\qquad B_+=\{i\in B:y_i>0\}
\]
are disjoint nonempty members of $\mathcal F$, and
$y(A_+)=y(B_+)=h$. By weak exchange, after interchanging $A_+$ and $B_+$ if
necessary, there is $b\in B_+$ such that
$A_+\cup\{b\}\in\mathcal F$. But
\[
y(A_+\cup\{b\})=y(A_+)+y_b=h+y_b>h,
\]
contrary to the inequality $y(S)\leq h$ for
$S\in\mathcal F$. Thus  
\begin{equation}\label{eq:rp-same-signs}
 p_A,p_B\in Q\quad\Longrightarrow\quad A\cap B\ne\varnothing.
\end{equation}
Applying the same argument to $-Q$ gives the corresponding statement
for two negative vertices of~$Q$.

Now let $Q,Q'$ be facets containing $p_A,-p_A$, respectively.
If they intersect, their intersection is a nonempty face and contains
a vertex $\varepsilon p_B$, where $\varepsilon\in\{1,-1\}$.
One of these facets then contains two vertices with the same sign,
and the other contains two vertices with opposite signs, with the
same supports $A,B$. Equations~\eqref{eq:rp-opposite-signs}
and~\eqref{eq:rp-same-signs} give a contradiction. Thus $Q\cap Q'$
is empty, and Lemma~\ref{lem:rp-quotient} completes the proof.
\end{proof}

\paragraph{Weakly exchanging families from cographs.}
For graphs on disjoint vertex sets, write $G\sqcup H$ for their
disjoint union and $G\vee H$ for their \emph{join}, obtained by adding
all edges between the two vertex sets. A \emph{cograph} is a graph
obtained from one-vertex graphs by repeated disjoint unions and joins. 
Every induced subgraph of a cograph
is again a cograph. 
Throughout, maximal means
inclusion-maximal.

\begin{lemma}\label{lem:cograph-intersection}
In every nonempty cograph, every maximal clique meets every maximal
independent set.
\end{lemma}

\begin{proof}
Induct on the recursive construction. The assertion holds for a
one-vertex graph. In a disjoint union of two nonempty graphs, a
maximal clique lies in one factor and is maximal there, whereas a
maximal independent set restricts to a maximal independent set in
each factor. Induction gives the required intersection. In a join,
a maximal independent set lies in one factor and is maximal there,
whereas a maximal clique restricts to a maximal clique in each
factor. Induction again applies.
\end{proof}

For a graph $G$, let $c(G)$ and $i(G)$ count all its cliques and all
its independent sets, respectively, including the empty set in both
counts. Define
\begin{equation}\label{eq:rp-split-family}
 \mathcal F(G)=\{C\cup I:
       C\subseteq V(G)\text{ is a clique and }
       I\subseteq V(G)\text{ is independent}\}.
\end{equation}
Equivalently, $S\in\mathcal F(G)$ if $S$ can be partitioned into a
clique and an independent set: any overlap in $C\cup I$ can be
removed from $I$. Either part is allowed to be empty. In particular,
$\mathcal F(G)$ is atomic and downward closed, and
\begin{equation}\label{eq:rp-family-count}
 |\mathcal F(G)|\leq c(G)i(G).
\end{equation}

\begin{lemma}\label{lem:cograph-weak-exchange}
If $G$ is a cograph, then $\mathcal F(G)$ satisfies weak exchange.
\end{lemma}

\begin{proof}
Suppose disjoint nonempty $A,B\in\mathcal F(G)$ violate weak exchange.
Choose partitions $A=C_A\mathbin{\dot\cup}I_A$ and
$B=C_B\mathbin{\dot\cup}I_B$ into clique and independent parts, and
work in the induced cograph $H=G[A\cup B]$.

Extend $C_A$ to a maximal clique $C$ of $H$. If $b\in C\cap B$,
then $A\cup\{b\}$ is the disjoint union of the clique
$C_A\cup\{b\}$ and the independent set $I_A$, contrary to our
assumption. Therefore $C\subseteq A$.

Likewise, extend $I_B$ to a maximal independent set $I$ of $H$.
If $a\in I\cap A$, then $B\cup\{a\}$ is the disjoint union of
$C_B$ and $I_B\cup\{a\}$, again a contradiction. Thus $I\subseteq B$.
The maximal clique $C$ and maximal independent set $I$ of $H$ are
disjoint, contradicting Lemma~\ref{lem:cograph-intersection}.
\end{proof}

\begin{proof}[Proof of Theorem~\ref{thm:rp-upper}]
Starting with $G_0=K_1$, recursively define
\begin{equation}\label{eq:rp-cograph-recursion}
 G_{k+1}=(G_k\sqcup G_k)\vee(G_k\sqcup G_k),
\end{equation}
where all four copies have disjoint vertex sets. Thus $G_k$ is a
cograph on $4^k$ vertices. Put $c_k=c(G_k)$ and $i_k=i(G_k)$.
For any two graphs on disjoint vertex sets,
\[
 \begin{aligned}
 c(G\sqcup H)&=c(G)+c(H)-1,&
 i(G\sqcup H)&=i(G)i(H),\\
 c(G\vee H)&=c(G)c(H),&
 i(G\vee H)&=i(G)+i(H)-1.
 \end{aligned}
\]
Consequently,
\begin{equation}\label{eq:rp-count-recursion}
 c_0=i_0=2,\qquad
 c_{k+1}=(2c_k-1)^2,\qquad i_{k+1}=2i_k^2-1.
\end{equation}
Using $c_{k+1}\leq4c_k^2$ and $i_{k+1}\leq2i_k^2$, induction gives
\begin{equation}\label{eq:rp-count-bounds}
 c_k\leq2^{3\cdot2^k-2},\qquad
 i_k\leq2^{2\cdot2^k-1},\qquad
 c_ki_k\leq2^{5\cdot2^k-3}.
\end{equation}

For an arbitrary $n\geq1$, choose the least $k\geq0$ with $4^k\geq n$,
and let $H$ be any induced $n$-vertex subgraph of $G_k$. It is a
cograph, and every clique or independent set of $H$ is one of $G_k$.
Lemmas~\ref{lem:cograph-weak-exchange}
and~\ref{lem:weak-exchange-polytope}, together with
\eqref{eq:rp-family-count}, give
\[
 \Delta(\RP^{n-1})\leq |\mathcal F(H)|-1
 \leq c(H)i(H)-1
 \leq c_ki_k-1
 \leq2^{5\cdot2^k-3}-1.
\]
Minimality of $k$ implies $4^k<4n$, so $2^k<2\sqrt n$ and
$\Delta(\RP^{n-1})<2^{10\sqrt n}$. Taking $n=d+1$ proves the upper bound.
 \end{proof}

\section{Low-dimensional factorization and vertex bounds}
\label{app:lie-groups}\label{app:compression-applications}

The Morse method also gives low-dimensional factorizations of maps
into spaces whose homotopy groups are concentrated in bounded degrees.
Recall that a CW complex is an \emph{$r$-type} if its homotopy groups
vanish above degree $r$. Write $\eta_r\colon X\to P_rX$ for the
Postnikov map, which preserves components and homotopy groups through
degree $r$.

\thmpostnikovcompression*
We first extend maps into $r$-types, up to homotopy, to a
suitable low-energy neighborhood. We then apply
\Cref{lem:sampled-energy-compression} to factor the inclusion
into that neighborhood through a low-dimensional CW complex.

\subsection{Maps into truncated targets}
\label{app:cohomological-compression}

Precomposition with the Postnikov map induces a bijection
\[
 [P_rX,T]\xrightarrow{\ \cong\ }[X,T]
 \qquad\text{for every CW $r$-type }T;
\]
see~\cite[Sections~4.1 and~4.3]{Hatcher}. Throughout this section,
brackets denote unbased homotopy classes.  For disconnected spaces, Postnikov
sections are taken componentwise.

\begin{lemma}[Maps into truncated targets]\label{lem:truncated-targets}
Let $e\colon A\to B$ be an $(r+1)$-connected map between spaces of
CW homotopy type. Then precomposition induces a bijection
$e^*\colon[B,T]\to[A,T]$ for every CW $r$-type $T$.
\end{lemma}

\begin{proof}
The induced map $P_re\colon P_rA\to P_rB$ is a weak equivalence:
$e$ induces isomorphisms on components and homotopy groups through
degree $r$, while both Postnikov sections have trivial homotopy
groups in higher degrees. Whitehead's theorem, applied to CW models,
and the universal property of Postnikov sections give the result.
Alternatively, relative CW approximation replaces $e$ by the attachment
of cells of dimension at least $r+2$. Maps into an $r$-type extend
across these cells uniquely up to homotopy
\cite[Sections~4.1 and~4.3]{Hatcher}.
\end{proof}

\subsection{Extension to a low-energy neighborhood}
\begin{lemma}[Extension to a low-energy neighborhood]
\label{lem:low-energy-extension}
Fix $r\geq1$ and set $k=r+2$. There is a constant
$\varepsilon_r>0$, depending only on $r$, with the following property.
Let $P$ be a finite simplicial complex on $[N]$ whose minimal
nonfaces have at most $k$ vertices. On the probability simplex
$\Delta\subseteq\R^N$, define
\begin{equation}\label{eq:nonface-energy}
 E(p)=\sum_{i_1,\ldots,i_k\in[N]}
 b_{i_1\ldots i_k}p_{i_1}\cdots p_{i_k},
 \qquad
 b_{i_1\ldots i_k}=
 \begin{cases}
  1,&\{i_1,\ldots,i_k\}\notin P,\\
  0,&\{i_1,\ldots,i_k\}\in P.
 \end{cases}
\end{equation}
Then $E^{-1}(0)=|P|$. Moreover, for
\[
 Y=\{p\in\Delta:E(p)<\varepsilon_r\},
\]
restriction along $|P|\hookrightarrow Y$ induces a surjection
\[
 [Y,T]\longrightarrow[|P|,T]
\]
for every CW $r$-type $T$. The same constant $\varepsilon_r$ works
for all $P$, $N$, and $T$.
\end{lemma}

\begin{proof}
The coefficients in \eqref{eq:nonface-energy} are symmetric in their
indices, and $E(p)$ is the probability that the set of values taken
by $k$ independent samples from $p$ is not a face of~$P$.
In particular, $E$ vanishes on~$|P|$. Conversely, if the support of
$p$ is not a face, it contains a minimal nonface of size at most~$k$.
The samples include this nonface with positive probability, allowing
repetitions when necessary. Thus $E(p)>0$, proving
$E^{-1}(0)=|P|$.

Put
\begin{equation}\label{eq:compression-constants}
 M=2r+2,\qquad
 B_r=\sum_{s=0}^{k-1}\binom Ms,\qquad
 \delta=\frac1{2B_r},\qquad
 \varepsilon_r=\delta^k.
\end{equation}
For $\eta\subseteq[N]$ with $s=|\eta|\leq k$, let
\[
 E_\eta(p)=
 \Pr\bigl[\eta\cup\{I_1,\ldots,I_{k-s}\}\notin P\bigr],
\]
where the $I_i$ are independent samples from $p$. Thus
$E_\varnothing=E$, and when $s=k$, $E_\eta$ is the indicator that
$\eta$ is a nonface. More generally, $E_\eta=1$ whenever
$\eta\notin P$. Conditioning on the first sample and discarding the
nonnegative terms for which that sample lies in $\eta$ gives
\begin{equation}\label{eq:conditional-nonface-energy}
 \sum_{v\notin\eta}p_vE_{\eta\cup\{v\}}(p)
 \leq E_\eta(p)
 \qquad(|\eta|<k).
\end{equation}

For $p\in Y$, define a simplicial complex $C_p$ by
\begin{equation}\label{eq:good-face-complex}
 \sigma\in C_p
 \quad\Longleftrightarrow\quad
 E_\eta(p)<\delta^{k-|\eta|}
 \text{ for every }\eta\subseteq\sigma\text{ with }|\eta|\leq k.
\end{equation}
This condition is downward closed, so it defines a simplicial
complex. Moreover, $C_p\subseteq P$: every nonface of $P$ contains
a minimal nonface $\eta$ of size at most $k$, and $E_\eta=1$
violates \eqref{eq:good-face-complex}.

\paragraph{Connectivity of $C_p$.}
We show that every subcomplex $L\subseteq C_p$ on at most $M$ vertices
is contained in the closed star of a vertex of $C_p$. Choose a vertex
$v$ at random with distribution $p$. For every face $\eta\in L$ with
$|\eta|\leq k-1$, including the empty face,
\eqref{eq:conditional-nonface-energy} and Markov's inequality give
\[
 \Pr\bigl[v\notin\eta,\
 E_{\eta\cup\{v\}}(p)\geq\delta^{k-|\eta|-1}\bigr]
 \leq\frac{E_\eta(p)}{\delta^{k-|\eta|-1}}<\delta.
\]
There are at most $B_r$ such faces, so the union bound and
$B_r\delta=1/2$ show that some vertex $v$ avoids all these events.
For this choice, $\sigma\cup\{v\}\in C_p$ for every $\sigma\in L$:
each subset of size at most $k$ either lies in $\sigma$ or has the
form $\eta\cup\{v\}$ covered by the preceding inequalities. Thus
$L$ lies in the closed star of $v$. Taking $L=\{\varnothing\}$ also
shows that $C_p$ is nonempty.

In Barmak's terminology, $C_p$ is therefore $(2r+2)$-conic.
His conicity theorem implies that $C_p$ is $r$-connected
\cite[Appendix, Theorem~4]{Farber2024}.

\paragraph{The incidence space.}
Let $h\colon|\operatorname{sd}P|\to|P|$ be the canonical homeomorphism
from the barycentric subdivision, and define
\[
 \mathcal R=\{(p,z)\in Y\times|\operatorname{sd}P|:
                    z\in|\operatorname{sd}C_p|\}.
\]
Consider the maps
\[
 \pi\colon \mathcal R\longrightarrow Y,\quad (p,z)\longmapsto p,
 \qquad
 \rho\colon \mathcal R\longrightarrow|P|,\quad (p,z)\longmapsto h(z).
\]
The fiber of $\pi$ over $p$ is $|\operatorname{sd}C_p|$, hence is
$r$-connected. We next verify that $\pi$ is a Serre microfibration:
a lift at time zero of a homotopy parametrized by a disk extends
for a sufficiently short time.

For each nonempty face $\sigma\in P$, choose a continuous function
$\varphi_\sigma\colon Y\to[0,\infty)$ that is positive precisely
when $\sigma\in C_p$. For example, take
\[
 \varphi_\sigma(p)=
 \max\left\{0,\min_{\substack{\eta\subseteq\sigma\\|\eta|\leq k}}
          \bigl(\delta^{k-|\eta|}-E_\eta(p)\bigr)\right\}.
\]
Consider a homotopy $p_t(x)$ in $Y$, with $x$ in a compact disk,
and an initial lift with barycentric coordinates $w_\sigma(x)$ in
$|\operatorname{sd}P|$. Set
\[
 \widetilde w_\sigma(x,t)=
 \begin{cases}
 w_\sigma(x)\min\left\{1,
 \dfrac{\varphi_\sigma(p_t(x))}{\varphi_\sigma(p_0(x))}\right\},
       &w_\sigma(x)>0,\\[6pt]
 0,    &w_\sigma(x)=0.
 \end{cases}
\]
The denominator is positive whenever $w_\sigma(x)>0$. Each
$\widetilde w_\sigma$ is continuous, including where $w_\sigma=0$,
since $0\leq\widetilde w_\sigma\leq w_\sigma$. Their sum is one at
$t=0$ and, by compactness, remains positive for a uniform short time.
After normalization by this sum, the support is a subchain of the
original chain of faces, and every face with positive weight belongs
to $C_{p_t(x)}$. The normalized weights therefore define the
required lift.

The microfibration connectivity theorem
\cite[Theorem~1.3]{Raptis2017}, applied to the identity fibration of
$Y$, now implies that $\pi$ is $(r+1)$-connected. Indeed, each fiber
map $|\operatorname{sd}C_p|\to\{p\}$ is $(r+1)$-connected.
The set $Y$ is semialgebraic because it is defined by the
polynomial inequality $E(p)<\varepsilon_r$ on $\Delta$.
To see that $\mathcal R$ is semialgebraic, fix a simplex
of $\operatorname{sd}P$. For $z$ in its relative interior,
the condition $z\in|\operatorname{sd}C_p|$ is given by the
polynomial inequalities in \eqref{eq:good-face-complex}
for the faces in the corresponding chain. Since there
are only finitely many simplices, $\mathcal R$ is
semialgebraic as well. Both spaces therefore have CW
homotopy type, and Lemma~\ref{lem:truncated-targets}
gives a bijection
\begin{equation}\label{eq:incidence-maps}
 \pi^*\colon[Y,T]\xrightarrow{\ \cong\ }[\mathcal R,T]
\end{equation}
for every CW $r$-type $T$.

For $p\in|P|$, every subset $\eta$ of its support with
$|\eta|\leq k$ satisfies $E_\eta(p)=0$. Thus the simplex supporting
$p$ belongs to $C_p$, and
\[
 s\colon |P|\longrightarrow\mathcal R,\qquad
 s(p)=(p,h^{-1}(p))
\]
is well defined. If $\iota\colon |P|\hookrightarrow Y$ is the inclusion,
then $\pi s=\iota$ and $\rho s=\mathrm{id}_{|P|}$.
Given $f\colon|P|\to T$, the bijection \eqref{eq:incidence-maps}
provides a map $g\colon Y\to T$ with $g\pi\simeq f\rho$. Hence
\[
 g\iota=g\pi s\simeq f\rho s=f,
\]
which proves the extension statement.
\end{proof}

\subsection{A sampling lemma and proof of \Cref{thm:postnikov-compression}}

The next lemma supplies the required low-dimensional
factorization. We state it for finite group actions, since
its equivariant form will also be used in the VC-dimension
bound.
\begin{lemma}[Sampling lemma]
\label{lem:sampled-energy-compression}
Let a finite group $G$ act on a finite set $V$ with $v=|V|\geq2$,
and let $k\geq2$ be an integer. Suppose
$b:V^k\to[0,1]$ is symmetric in its arguments and invariant under
the diagonal $G$-action. On the probability simplex $\Delta(V)$,
define
\(
 E(p)=\mathbb E\,b(X_1,\ldots,X_k),
\)
where $X_1,\ldots,X_k$ are independent samples with law $p$.

Let $0<\varepsilon\leq1$, and suppose that the induced $G$-action
on $Y=\{p\in\Delta(V):E(p)<\varepsilon\}$ is free.
Then the inclusion $\iota:E^{-1}(0)\hookrightarrow Y$ factors
$G$-equivariantly up to homotopy as
\[
 E^{-1}(0)\xrightarrow{\alpha}Z\xrightarrow{\beta}Y,
 \qquad
 \beta\alpha\simeq_G\iota,
\]
where $Z$ is a finite free $G$-CW complex satisfying
\[
 \dim Z\leq
 \left\lfloor
 2k(k-1)\varepsilon^{-2}\log^2v
 \right\rfloor.
\]
For the trivial group, the freeness assumption is automatic.
\end{lemma}

\begin{proof}
Write $S=E^{-1}(0)$. If $S$ is empty, there is nothing to prove.

\paragraph{The Hessian estimate.}
Fix $p\in\Delta^\circ$ and put $B=b(X_1,\ldots,X_k)$. Define
\[
 \begin{split}
 g(i,j)&=\mathbb E[B\mid X_1=i,X_2=j],\\
 g_1(i)&=\mathbb E[B\mid X_1=i],\qquad
 \mu=\mathbb EB,\\
 h(i,j)&=g(i,j)-g_1(i)-g_1(j)+\mu.
 \end{split}
\]
Symmetry gives
\(
 \mathbb E[h(X_1,X_2)\mid X_1]
 =\mathbb E[h(X_1,X_2)\mid X_2]=0.
\)
Consequently, the variables $h(X_a,X_b)$, $a<b$, are pairwise
orthogonal. For $U=\sum_{a<b}h(X_a,X_b)$, conditioning on each
pair also gives
\[
 \mathbb EU^2
 =\binom{k}{2}\mathbb Eh(X_1,X_2)^2
 =\mathbb E[(B-\mu)U].
\]
Expanding $\mathbb E(B-\mu-U)^2\geq0$, we obtain
\[
 \binom{k}{2}\mathbb Eh(X_1,X_2)^2
 \leq\operatorname{Var}(B)\leq\frac14.
\]

Set
\(
 D_p=\diag(\sqrt p),\
 u_p=(\sqrt{p_i})_{i\in V},\
 \Pi_p=I-u_pu_p^{\mathsf T}.
\)
Differentiating the symmetric degree-$k$ polynomial $E$ gives,
including its diagonal second derivatives,
\(
 \nabla^2E(p)=k(k-1)(g(i,j))_{i,j}.
\)
Multiplying out the projections therefore yields
\[
 R_p:=\Pi_pD_p\nabla^2E(p)D_p\Pi_p
     =k(k-1)D_p(h(i,j))_{i,j}D_p.
\]
Hence
\begin{equation}\label{eq:sampled-energy-hessian}
 \|R_p\|_F^2
 =k^2(k-1)^2\mathbb Eh(X_1,X_2)^2
 \leq\frac{k(k-1)}2.
\end{equation}

\paragraph{A compact superlevel inside $Y$.}
Let $u=(1/v,\ldots,1/v)$ and put
\(
 a=\frac{\varepsilon}{2\log v},\
 \lambda=\frac{\varepsilon}{16k},\
 \eta=\frac{\varepsilon}{16v\log(v/\lambda)}.
\)
On $\Delta^\circ$, define the smooth invariant function
\[
 F(p)=-E(p)+aH(p)+\eta\sum_{i\in V}\log p_i,
 \qquad
 H(p)=-\sum_{i\in V}p_i\log p_i.
\]
The logarithmic barrier tends to $-\infty$ at the boundary,
so every superlevel of $F$ is compact in $\Delta^\circ$.
Since $H\leq\log v$ and the barrier is nonpositive,
\[
 \{F\geq-\varepsilon/4\}
 \subseteq\{E\leq3\varepsilon/4\}
 \subseteq Y.
\]
Thus the action on this superlevel is free.

For $0\leq t\leq1$, define the equivariant inward homotopy
\(
 j_{t\lambda}(p)=(1-t\lambda)p+t\lambda u.
\)
For $p\in S$, generate each sample independently from $p$
with probability $1-t\lambda$ and from $u$ otherwise.
If every sample uses $p$, the expected kernel value is
$E(p)=0$; in all other cases it is at most one. Therefore
\[
 E(j_{t\lambda}(p))
 \leq1-(1-t\lambda)^k
 \leq kt\lambda
 \leq\varepsilon/16.
\]
In particular, the homotopy stays in $Y$.
Every coordinate of $j_\lambda(p)$ is at least $\lambda/v$,
so the barrier term there is at least $-\varepsilon/16$.
Together with $H\geq0$, this gives
\(
 F(j_\lambda(p))\geq-\varepsilon/8
\) for  \(p\in S\).

\paragraph{The dimension bound and factorization.}
Compute Hessians in the affine coordinates of $\Delta^\circ$.
Under the substitution $z=D_pw$, the tangent condition
$\sum_i z_i=0$ becomes $w\perp u_p$. On this subspace the
Hessian quadratic form of $F$ is represented by the restriction of
\[
 -R_p-aI-\eta\diag(1/p)\preceq-R_p-aI.
\]
By the spectral theorem and \eqref{eq:sampled-energy-hessian},
\[
 \begin{split}
 \nuplus(\Hess_{\Delta^\circ}F(p))
 &\leq
 \#\{\theta\in\operatorname{spec}(R_p):\theta\leq-a\}\\
 &\leq a^{-2}\|R_p\|_F^2\\
 &\leq2k(k-1)\varepsilon^{-2}\log^2v.
 \end{split}
\]
Apply Lemma~\ref{lem:morse} with lower level $-\varepsilon/4$,
upper level $-\varepsilon/8$, and
\(
 D=\left\lfloor
 2k(k-1)\varepsilon^{-2}\log^2v
 \right\rfloor.
\)
It gives an invariant subspace $W$ with
\(
 j_\lambda(S)\subseteq W
 \subseteq\{F\geq-\varepsilon/4\}\subseteq Y
\)
and $G$-equivariant homotopy inverses
$e:W\to Z$ and $\ell:Z\to W$, where $Z$ is a finite free
$G$-CW complex of dimension at most $D$.

Define
\(
 \alpha=e\circ j_\lambda|_S,\
 \beta=\operatorname{incl}_{W\hookrightarrow Y}\circ\ell.
\)
Then $\beta\alpha$ is equivariantly homotopic in $Y$ to
$j_\lambda|_S$, which is equivariantly homotopic to $\iota$
by the inward homotopy above.
\end{proof}

\begin{proof}[Proof of Theorem~\ref{thm:postnikov-compression}]
Let $k=r+2$, and enlarge $K$ to the complex
\[
 P=\{\sigma\subseteq[N]:
       \eta\in K\text{ for every }\eta\subseteq\sigma
                       \text{ with }|\eta|\leq k\}.
\]
Every minimal nonface of $P$ has at most $k$ vertices, and $K$ and
$P$ have the same $(r+1)$-skeleton. Thus the inclusion
$\iota_K\colon|K|\hookrightarrow|P|$ adds only cells of dimension at
least $r+2$ and is $(r+1)$-connected. By
Lemma~\ref{lem:truncated-targets}, the Postnikov map
$\eta_r\colon|K|\to P_r|K|$ extends up to homotopy to a map
$f_P\colon|P|\to P_r|K|$.

Let $E$ and $Y$ be as in Lemma~\ref{lem:low-energy-extension}, and
write $\varepsilon=\varepsilon_r$. That lemma gives a further
extension $g\colon Y\to P_r|K|$ satisfying $g\iota\simeq f_P$, where
$\iota\colon|P|\hookrightarrow Y$ is the inclusion.  

\paragraph{The factorization.}
Apply Lemma~\ref{lem:sampled-energy-compression} with the
trivial group, $v=N$, $k=r+2$, and $\varepsilon=\varepsilon_r$.
Since $E^{-1}(0)=|P|$, it gives maps
\[
 |P|\xrightarrow{\alpha}Z\xrightarrow{\beta}Y,
 \qquad
 \beta\alpha\simeq\iota,
 \qquad
 \dim Z\leq C_r\log^2N,
 \qquad
 C_r=\frac{2(r+2)(r+1)}{\varepsilon_r^2}.
\]
Since $B_r=\sum_{s=0}^{r+1}\binom{2r+2}{s}\leq2^{2r+2}$ and
$\varepsilon_r^{-2}=(2B_r)^{2(r+2)}$, we obtain $C_r=2^{O(r^2)}$.
Set $j=\alpha\iota_K$ and $v=g\beta$. Then
\[
 vj
 \simeq g\iota\iota_K
 \simeq f_P\iota_K
 \simeq\eta_r.
\]
For a map $f\colon |K|\to T$ to any CW $r$-type, the universal
property gives $\widehat f\colon P_r|K|\to T$ with
$f\simeq\widehat f\eta_r$. Thus $\bar f=\widehat f v$ satisfies
$f\simeq\bar f j$, using the same $j$ for every target and map.
If $K$ is connected, retain only the component of $Z$
containing $j(|K|)$.
\end{proof}

\subsection{Vertex bounds from Postnikov factorization}\label{app:essential-complexes}

\begin{corollary}[Detection in an $r$-type]
\label{cor:truncated-detection}\label{cor:essential-vertices}
Let $r,d\geq1$ be integers, let $K$ be a finite simplicial complex on
$N\geq2$ vertices, and let $f\colon|K|\to T$ be a map to a CW $r$-type.
If $f$ does not factor
up to homotopy through any CW complex of dimension less than $d$,
then
\begin{equation}\label{eq:truncated-detection}
 d\leq C_r\log^2N,\qquad
 N\geq\exp\!\left(\sqrt{d/C_r}\right).
\end{equation}
The hypothesis holds, in particular, if $f^*\gamma\ne0$ for some
$\gamma\in H^d(T;A)$ with constant abelian coefficients $A$, or if
$f$ is not homotopic to a map into $T^{(d-1)}$.
\end{corollary}

\begin{proof}
Theorem~\ref{thm:postnikov-compression} factors $f$ through a finite
CW complex of dimension at most $C_r\log^2N$, so the hypothesis
implies $d\leq C_r\log^2N$.

Suppose $f$ factors up to homotopy as
$|K|\to Z'\to T$, where $\dim Z'=D<d$.
Since $H^d(Z';A)=0$, we have $f^*\gamma=0$ for every
$\gamma\in H^d(T;A)$. By cellular approximation, the map
$Z'\to T$ is homotopic to one with image in
$T^{(D)}\subseteq T^{(d-1)}$, so $f$ is also homotopic
to a map into $T^{(d-1)}$. Either of the stated
conditions therefore rules out such a factorization.  
\end{proof}

\paragraph{Cohomology, products, and operations.}
For an abelian group $A$ and $1\leq q\leq r$, the Eilenberg--MacLane
space $K(A,q)$ is an $r$-type and represents $H^q(-;A)$
\cite[Theorem~4.57]{Hatcher}. Consequently, for the same map $j$ in
Theorem~\ref{thm:postnikov-compression}, the homomorphism
\[
 j^*\colon H^q(Z;A)\longrightarrow H^q(K;A)
\]
is surjective for every $0\leq q\leq r$; for $q=0$, use a discrete
representing space.  

\begin{corollary}[Products of bounded-degree cohomology classes]
\label{cor:bounded-degree-products}\label{cor:steenrod-vertices}
Let $K$ be a finite simplicial complex with $N\geq2$ vertices,
let $r\geq1$ be an integer, and let $R$ be a commutative ring
with identity. Suppose that
\[
 \alpha_i\in H^{q_i}(K;R),
 \qquad 1\leq q_i\leq r,
 \qquad i=1,\ldots,s,
\]
satisfy
\(
 \alpha_1\smile\cdots\smile\alpha_s\ne0.
\)
Then
\(
 q_1+\cdots+q_s\leq C_r\log^2N,
 \) and hence 
 \(
 N\geq
 \exp\!\left(\sqrt{\frac{q_1+\cdots+q_s}{C_r}}\right),
\)
where $C_r$ is the constant from
Theorem~\ref{thm:postnikov-compression}.  
Equivalently, the graded subalgebra of $H^*(K;R)$ generated by
classes of degrees $1,\ldots,r$ vanishes in every degree greater
than $C_r\log^2N$.
\end{corollary}

\begin{proof}
Let $j\colon |K|\to Z$ be the map supplied by
Theorem~\ref{thm:postnikov-compression}, with
\(
 \dim Z\leq C_r\log^2N.
\)
As observed above, the homomorphism
\[
 j^*\colon H^q(Z;R)\longrightarrow H^q(K;R)
\]
is surjective for every $1\leq q\leq r$. Choose
$\widetilde\alpha_i\in H^{q_i}(Z;R)$ with
$j^*\widetilde\alpha_i=\alpha_i$. Naturality of the cup product
gives
\[
 j^*\!\left(
   \widetilde\alpha_1\smile\cdots\smile\widetilde\alpha_s
 \right)
 =
 \alpha_1\smile\cdots\smile\alpha_s
 \ne0.
\]
Thus the product on $Z$ is nonzero. Since its degree is
$q_1+\cdots+q_s$, we obtain
\[
 q_1+\cdots+q_s
 \leq\dim Z
 \leq C_r\log^2N.
 \qedhere
\]
\end{proof}

\paragraph{Natural cohomology operations.}
The same argument applies to any operation that is defined on
every tuple of input classes and commutes with pullback.
More precisely, fix abelian coefficient groups
$A_1,\ldots,A_s,B$ and integers $1\leq q_i\leq r$.
Suppose that for every CW complex $X$ there is an operation
\[
 \Theta_X\colon
 \prod_{i=1}^s H^{q_i}(X;A_i)
 \longrightarrow H^d(X;B),
\]
defined on all input tuples, such that every continuous map
$h\colon X\to Y$ satisfies
\[
 h^*\Theta_Y(a_1,\ldots,a_s)
 =
 \Theta_X(h^*a_1,\ldots,h^*a_s).
\]
All coefficient groups here are constant. Then
\[
 \Theta_{|K|}(\alpha_1,\ldots,\alpha_s)\ne0
 \quad\Longrightarrow\quad
 d\leq C_r\log^2N.
\]

Indeed, the same map $j\colon |K|\to Z$ induces surjections
in degrees $1,\ldots,r$ for every abelian coefficient group.
Choose lifts
$\widetilde\alpha_i\in H^{q_i}(Z;A_i)$ of the input classes.
Naturality gives
\[
 \Theta_{|K|}(\alpha_1,\ldots,\alpha_s)
 =
 j^*\Theta_Z(\widetilde\alpha_1,\ldots,\widetilde\alpha_s).
\]
Nonvanishing therefore forces $H^d(Z;B)\ne0$, and hence
$d\leq\dim Z\leq C_r\log^2N$.

A finite expression formed using cup products and operations
of this kind is itself such a natural operation on its initial
classes. Thus only the initial classes must have degrees at
most $r$; intermediate classes may have higher degrees.
For example, if
\[
 \alpha\in H^q(K;\mathbb F_2),
 \qquad 1\leq q\leq r,
 \qquad
 \operatorname{Sq}^{i_1}\cdots
 \operatorname{Sq}^{i_t}\alpha\ne0,
\]
then
\[
 q+i_1+\cdots+i_t\leq C_r\log^2N.
\]

The requirement that the operation be defined on every input
tuple is essential to this argument. For a partially defined
operation, such as a Massey product, the chosen lifts on $Z$
need not satisfy the vanishing relations required to define
the operation.

\paragraph{Classifying maps and essentiality.}
For connected $K$, its first Postnikov section is
$B\pi_1(K)=K(\pi_1(K),1)$, and $\eta_1$ is its classifying map $c_K$.
Theorem~\ref{thm:postnikov-compression}, with $r=1$, therefore gives
\begin{equation}\label{eq:classifying-compression}
 |K|\xrightarrow{j}Z\xrightarrow{v}B\pi_1(K),
 \qquad vj\simeq c_K,\quad\dim Z\leq C_1\log^2N.
\end{equation}
Following \cite{ABHK2024}, $K$ is \emph{topologically $d$-essential}
if $c_K$ cannot be deformed into the $(d-1)$-skeleton of
$B\pi_1(K)$. Corollary~\ref{cor:truncated-detection} therefore gives
$N\geq\exp(\sqrt{d/C_1})$. Here topological essentiality is distinct
from the weaker combinatorial notion in \cite{ABHK2024}: the
complete-graph examples in \cite[Remark~1.17]{ABHK2024} are
combinatorially $d$-essential but have only linearly many vertices.

\begin{corollary}[Highly connected finite covers]
\label{cor:finite-cover-vertices}
Let $K$ be a connected finite simplicial complex on $N\geq2$ vertices.
Suppose $E\to|K|$ is a principal $\Gamma$-cover, where $\Gamma$ is a
nontrivial finite group and $E$ is $(d-1)$-connected for some $d\geq1$.
Then
\begin{equation}\label{eq:finite-cover-vertices}
 d\leq C_1\log^2N,\qquad
 N\geq\exp\!\left(\sqrt{d/C_1}\right).
\end{equation}
The constant is independent of $\Gamma$.
\end{corollary}

\begin{proof}
Let $f\colon |K|\to B\Gamma$ classify the cover. Since $B\Gamma$ is a
$1$-type, Theorem~\ref{thm:postnikov-compression} gives
$f\simeq gj$ with $j\colon |K|\to Z$, $g\colon Z\to B\Gamma$, and
$\dim Z\leq C_1\log^2N$. Pull the universal $\Gamma$-bundle back
along $g$. Its total space $\widetilde Z$ is a finite free
$\Gamma$-CW complex of dimension $\dim Z$. The pullback cover $j^*\widetilde Z\to|K|$ is classified
by $gj$, which is homotopic to $f$. It is therefore
isomorphic to $E\to|K|$. Composing an isomorphism
$E\to j^*\widetilde Z$ with the natural map
$j^*\widetilde Z\to\widetilde Z$ gives a
$\Gamma$-equivariant map $E\to\widetilde Z$.

Set $D=\dim Z=\dim\widetilde Z$. The join
$\Gamma^{*(D+1)}$ is a free $D$-dimensional simplicial
$\Gamma$-complex and is $(D-1)$-connected. We can therefore map
$\widetilde Z$ equivariantly to this join, extending one orbit of
cells at a time. If $D<d$, then $E$ is $D$-connected, so the composite
$E\to\widetilde Z\to\Gamma^{*(D+1)}$ contradicts Dold's
theorem~\cite{Dold1983}. The covering $E\to|K|$ gives $E$ a finite simplicial
structure, so the simplicial formulation of Dold's
theorem applies.
Hence $d\leq D\leq C_1\log^2N$.
\end{proof}

\subsection{Geometric applications}
\label{app:geometric-applications}

\corgeometric*

\begin{proof}[Proof of Corollary~\ref{cor:geometric-factorization}]
Let $K$ be a finite simplicial complex homotopy equivalent to the
space under consideration, and let $N$ be its number of vertices.

\smallskip
\noindent\emph{Projective spaces and symplectic manifolds.}
The cohomology generators of $\RP^d$, $\mathbb{CP}^m$, and
$\mathbb{HP}^m$ have degrees $1$, $2$, and $4$, respectively, with
nonzero powers in the top degree. Here we use mod-two coefficients
in the real case and integral coefficients in the other two
\cite[Corollary~3.6]{GMP}. For a closed symplectic $2m$-manifold
$(M,\omega)$, the symplectic orientation gives
$\int_M\omega^m>0$, so $[\omega]^m\ne0$. In each case,
Corollary~\ref{cor:bounded-degree-products} gives the required bound.
More generally, the same argument applies to any space $X$ with a
class $u\in H^2(X;\mathbb R)$ satisfying $u^m\ne0$.

\smallskip
\noindent\emph{Grassmannians, flags, and products.}
Complex Grassmannians and complete complex flag manifolds are smooth
projective varieties of complex dimensions $k(n-k)$ and $n(n-1)/2$,
respectively~\cite[Sections~1.1--1.2]{Brion2005}. The restriction of the Fubini--Study form is a symplectic
form on each of these manifolds, so the preceding argument
applies.

For a product of spheres, multiply the pullbacks of the top mod-two
classes of the factors. The cross-product description shows that
this product is nonzero~\cite[Section~3.2]{Hatcher}. Applying
Corollary~\ref{cor:bounded-degree-products} proves the bound for
products of spheres, including tori.

\smallskip
\noindent\emph{Real Grassmannians.}
Let $X=\operatorname{Gr}_{\mathbb R}(k,n)$ be the Grassmannian of
$k$-dimensional subspaces of $\mathbb R^n$, and let $\gamma_X$ be its
tautological bundle. Its real dimension is $D=k(n-k)$. The cases
$k=0,n$ are points, and the cases $k=1,n-1$ are real projective spaces.
Assume $k,n-k\geq2$, and set
\[
 a=\left\lfloor\frac{k}{2}\right\rfloor,\qquad
 b=\left\lfloor\frac{n-k}{2}\right\rfloor,\qquad
 \varepsilon=k-2a,\qquad \delta=n-k-2b.
\]
Thus $a,b\geq1$ and $\varepsilon,\delta\in\{0,1\}$.
Put $X_0=\operatorname{Gr}_{\mathbb R}(2a,2a+2b)$ and
$Y_0=\operatorname{Gr}_{\mathbb C}(a,a+b)$, with tautological bundles
$\gamma_0$ and $\xi$, respectively. Their standard rational cohomology
presentations give a degree-doubling ring isomorphism
\[
 \Phi\colon H^{2*}(Y_0;\mathbb Q)\xrightarrow{\ \cong\ }
                 H^{4*}(X_0;\mathbb Q),
 \qquad \Phi(c_i(\xi))=p_i(\gamma_0);
\]
see~\cite[Theorem~2.4]{He2017}. The class $-c_1(\xi)$ is the
hyperplane class of the Pl\"ucker embedding of $Y_0$, so its
$ab$th power is nonzero, since $\dim_{\mathbb C}Y_0=ab$
\cite[Section~1.1]{Brion2005}. Consequently,
\[
 p_1(\gamma_0)^{ab}=\Phi\bigl(c_1(\xi)^{ab}\bigr)\ne0.
\]

Choose a decomposition
$\mathbb R^n=\mathbb R^{2a+2b}\oplus\mathbb R^\varepsilon
\oplus\mathbb R^\delta$. The inclusion
\(
 \iota\colon X_0\longrightarrow X,\ 
 U\longmapsto U\oplus\mathbb R^\varepsilon\oplus\{0\}
\)
satisfies $\iota^*\gamma_X\cong
\gamma_0\oplus\underline{\mathbb R}^{\varepsilon}$. By naturality
and stability of Pontryagin classes,
\[
 \iota^*\bigl(p_1(\gamma_X)^{ab}\bigr)
 =p_1(\gamma_0)^{ab}\ne0.
\]
Hence $p_1(\gamma_X)^{ab}$ is a nonzero class of degree $4ab$.
Applying Corollary~\ref{cor:bounded-degree-products} with $r=4$ gives
$4ab\leq C_4\log^2N$. Finally,
$\lfloor t/2\rfloor\geq t/3$ for every integer $t\geq2$, so
\(
 \frac49D=\frac49k(n-k)\leq4ab\leq C_4\log^2N.
\)

\smallskip
\noindent\emph{Special orthogonal groups.}
The mod-two cohomology ring has the presentation
\[
 H^*(\operatorname{SO}(n);\mathbb F_2)
 \cong
 \mathbb F_2[x_j:1\leq j<n,\ j\text{ odd}]
 \big/\bigl(x_j^{2^{s_j}}:1\leq j<n,\ j\text{ odd}\bigr),
\]
where $|x_j|=j$ and $2^{s_j}$ is the least power of two such that
$j2^{s_j}\geq n$; see~\cite[proof of Proposition~2.11]{DMZ}.
In particular, for $s=\lceil\log_2n\rceil$ and $x=x_1$, we have
$x^{2^s-1}\ne0$. Since $2^s-1\geq n-1\geq n/2$ for $n\geq2$,
Corollary~\ref{cor:bounded-degree-products}, with $r=1$, gives the
required bound. Corollary~\ref{cor:so} below gives a better explicit
constant via the direct double-cover argument.

\smallskip
\noindent\emph{Unitary groups.}\label{app:unitary-groups}
The cohomology rings are
\begin{equation}\label{eq:unitary-cohomology}
 \begin{split}
 H^*(\operatorname{U}(n);\mathbb F_2)
   &=\Lambda(x_1,x_3,\ldots,x_{2n-1}),\\
 H^*(\operatorname{SU}(n);\mathbb F_2)
   &=\Lambda(x_3,x_5,\ldots,x_{2n-1}),\qquad |x_j|=j.
 \end{split}
\end{equation}
Choose these generators to be the cohomology suspensions of the
universal Chern classes. The Wu formula then gives
\begin{equation}\label{eq:unitary-steenrod}
 \operatorname{Sq}^{2i}x_{2j-1}
 =\binom{j-1}{i}x_{2(j+i)-1}
 \quad(0\leq i\leq j-1,\ j+i\leq n).
\end{equation}
The binomial coefficient is taken modulo two. Indeed, the
indecomposable term in $\operatorname{Sq}^{2i}c_j$ is
$\binom{j-1}{i}c_{j+i}$, and suspension kills decomposables and
commutes with Steenrod squares
\cite[Corollary~11.4 and Section~12]{BorelSerre1953}.
In particular,
\[
 x_3\xrightarrow{\operatorname{Sq}^2}x_5
 \xrightarrow{\operatorname{Sq}^4}x_9
 \xrightarrow{\operatorname{Sq}^8}x_{17}\longrightarrow\cdots.
\]
For $m=\lfloor\log_2(n-1)\rfloor$, this gives
\begin{equation}\label{eq:unitary-steenrod-word}
 \operatorname{Sq}^{2^m}\cdots\operatorname{Sq}^4
 \operatorname{Sq}^2x_3=x_{2^{m+1}+1}\ne0,
\end{equation}
where the composition is empty when $m=0$. The resulting nonzero
class has degree between $n$ and $2n-1$, so
Corollary~\ref{cor:steenrod-vertices}, with $r=3$, gives the bound.

\smallskip
\noindent\emph{Spherical quotients.}\label{app:lens-spaces}
Let $h\colon |K|\to S^d/\Gamma$ be a homotopy equivalence. Pulling back
$S^d\to S^d/\Gamma$ along $h$ gives a principal $\Gamma$-cover
$E\to|K|$ with $E\simeq S^d$; homotopy equivalences lift to homotopy
equivalences of these pullback covers~\cite[Section~1.3]{Hatcher}.
Corollary~\ref{cor:finite-cover-vertices} gives the required bound,
uniformly in $\Gamma$. For $d=1$, Corollary~\ref{cor:finite-cover-vertices}
still applies to $E\to|K|$, since it requires a principal
cover with path-connected total space, not a universal
cover.

\smallskip
\noindent\emph{Aspherical spaces.}
Suppose $|K|$ is connected and aspherical. Apply
Theorem~\ref{thm:postnikov-compression} to the identity map of the
$1$-type $|K|$. This gives a connected finite CW complex $Z$ and maps
\[
 |K|\xrightarrow{j}Z\xrightarrow{g}|K|,
 \qquad gj\simeq\mathrm{id}_{|K|},
 \qquad \dim Z\leq C_1\log^2N.
\]
For every local coefficient system $\mathcal L$ on $|K|$, the map
$g^*\colon H^t(|K|;\mathcal L)\to H^t(Z;g^*\mathcal L)$ is split
injective: the homotopy $gj\simeq\mathrm{id}$ identifies
$j^*g^*\mathcal L$ with $\mathcal L$ and makes $j^*g^*$ the identity
on cohomology. Its target vanishes for $t>\dim Z$, so
$\operatorname{cd}_{\mathbb Z}(\pi_1(K))\leq\dim Z$
\cite[Chapter~VIII]{Brown1982}. If $|K|\simeq M^d$ for a closed aspherical manifold $M$,
then $H^d(|K|;\mathbb F_2)\ne0$, whether or not $M$ is
orientable \cite[Section~3.3]{Hatcher}. The injectivity of
$g^*$ therefore gives $H^d(Z;\mathbb F_2)\ne0$, and hence
$d\leq\dim Z$.
\end{proof}

\paragraph{Configuration spaces and braid groups.}
The same arguments apply to configuration spaces. Let $q,n\geq2$,
let $u$ generate $H^{q-1}(S^{q-1};\mathbb F_2)$, and, for
$2\leq i\leq n$, set $\alpha_i=\varphi_i^*u$, where
\[
 \varphi_i\colon\operatorname{Conf}_n(\mathbb R^q)\to S^{q-1},
 \qquad
 \varphi_i(x_1,\ldots,x_n)=\frac{x_i-x_1}{\|x_i-x_1\|}.
\]
For fixed radii $0<r_2<\cdots<r_n$, the map
\[
 \psi\colon(S^{q-1})^{n-1}\longrightarrow\operatorname{Conf}_n(\R^q),
 \qquad (y_2,\ldots,y_n)\longmapsto(0,r_2y_2,\ldots,r_ny_n)
\]
is well defined: the points $r_i y_i$ have pairwise
distinct positive norms $r_i$, so they are distinct
from one another and from the origin. Moreover,
$\varphi_i\psi=\operatorname{pr}_i$ for every
$2\leq i\leq n$. Consequently,
$\psi^*(\alpha_2\cdots\alpha_n)$ is the nonzero top class of the
product of spheres. Corollary~\ref{cor:bounded-degree-products}
therefore gives
\[
 \operatorname{ct}(\operatorname{Conf}_n(\mathbb R^q))
 \geq
 \exp\!\left(\sqrt{\frac{(q-1)(n-1)}{C_{q-1}}}\right).
\]

The Fadell--Neuwirth fibrations show that planar ordered
configuration spaces are aspherical~\cite{FadellNeuwirth1962}.
Their unordered quotients have the same higher homotopy groups and
fundamental group $B_n$. Since
$\operatorname{cd}_{\mathbb Z}(B_n)=n-1$
\cite[Theorem~4.2]{LeonAlvarezMorales2025}, the preceding aspherical
estimate gives
\[
 \operatorname{ct}\bigl(\operatorname{Conf}_n(\mathbb R^2)/\mathfrak S_n\bigr)
 \geq\exp\!\left(\sqrt{\frac{n-1}{C_1}}\right).
\]

\subsection{The sharper degree-one estimate}
\label{subsec:cup-powers}

For degree-one cup powers, the double-cover argument gives
the following better bound directly from Theorem~\ref{thm:main},
without using Lemma~\ref{lem:low-energy-extension}.

\begin{corollary}[Degree-one cup powers]\label{cor:cup}
Let $T$ be a finite simplicial complex with $N$ vertices.
If $\alpha\in H^1(T;\mathbb F_2)$ satisfies $\alpha^h\ne0$
for an integer $h\geq1$, then
\begin{equation}\label{eq:cup}
 h\leq4\log^2(2N),
 \qquad\text{and hence}\qquad
 N\geq\frac12\exp\!\left(\frac{\sqrt h}{2}\right).
\end{equation}
In particular, \eqref{eq:rp} holds for every finite
simplicial complex homotopy equivalent to $\RP^d$,
$d\geq1$.
\end{corollary}

\begin{proof}
Let $K\to T$ be the double cover classified by $\alpha$,
equipped with its lifted simplicial structure and deck
involution. An equivariant map $|K|\to S^k$ induces a map
$|T|\to\RP^k$ under which the class of the double cover
$S^k\to\RP^k$ pulls back to $\alpha$. Since
$H^{k+1}(\RP^k;\mathbb F_2)=0$, this would imply
$\alpha^{k+1}=0$. Thus $\alpha^h\ne0$ forces $k\geq h$, and hence
$\ind|K|\geq h$.

The neighborhood enlargement in the proof of
Corollary~\ref{cor:rp} applies to this double cover. It contains $K$,
has $2N$ vertices and at most $2N$ facets, and carries a free
involution. Theorem~\ref{thm:main} therefore gives
\[
 h\leq\ind|K|\leq4\log^2(2N),
\]
which proves \eqref{eq:cup}. For the final assertion, pull back
along a homotopy equivalence the degree-one class of $\RP^d$ whose
$d$th power is nonzero.
\end{proof}

\begin{corollary}[Special orthogonal groups]\label{cor:so}
Let $n\geq2$ and put $s=\lceil\log_2 n\rceil$.
If a finite simplicial complex homotopy equivalent to
$\operatorname{SO}(n)$ has $N$ vertices, then
\begin{equation}\label{eq:so}
 N\geq
 \frac12\exp\!\left(\frac{\sqrt{2^s-1}}{2}\right)
 \geq
 \frac12\exp\!\left(\frac{\sqrt{n-1}}{2}\right)
\end{equation}
\end{corollary}

\begin{proof}
Apply Corollary~\ref{cor:cup} to the degree-one class satisfying
$x^{2^s-1}\ne0$ recalled in the proof of
Corollary~\ref{cor:geometric-factorization}. The second inequality
follows from $2^s\geq n$.
\end{proof}

\subsection{Computational topology}
\label{sec:computational-topology}

The degree bounds above give algorithms whose running time is
quasipolynomial in the size of a presentation by maximal simplices.
This presentation can be exponentially smaller than the list of
all simplices.

\begin{theorem}[Quasipolynomial algorithms for cohomological invariants]
\label{thm:computational-topology}
Let $K$ be a finite simplicial complex on $N\geq2$ vertices, given
by its maximal simplices, and let $L\geq2$ be the total input bit
length, including the involution when supplied.
\begin{enumerate}
\item Given a free simplicial involution $\tau$ on $K$, one can
compute its Stiefel--Whitney height in time
$\exp(O((\log L)^3))$. 

\item For each fixed $r\geq1$ and fixed finite field $\mathbb F$,
one can compute a graded basis, multiplication table, and
cup-length of the unital subalgebra
$A_r\subseteq H^*(K;\mathbb F)$ generated by classes of degrees
$1,\ldots,r$, in time $\exp(O_r((\log L)^3))$.
\end{enumerate}
In particular, the second algorithm computes the entire cohomology
ring and its cup-length whenever the ring is generated in these
degrees.
\end{theorem}

\begin{proof}
Write $f$ for the number of maximal simplices. In the first case,
Theorem~\ref{thm:main} gives
$h(K,\tau)\leq\ind|K|\leq4\log N\log f$; in the second,
Corollary~\ref{cor:bounded-degree-products} gives
$A_r^q=0$ for $q>C_r\log^2N$.
Choose $D=\lceil4\log N\log f\rceil$ or
$D=\lceil C_r\log^2N\rceil$, respectively.
Since $N,f\leq L$, we have $D=O_r((\log L)^2)$.
Enumerate all faces through dimension $D+1$; their number is
$L^{O(D)}$, and they determine cohomology and cup products
through degree $D$.

For the first algorithm, order the vertex orbits. No simplex
contains an antipodal pair, so retaining each simplex orbit as a
separate quotient simplex gives an ordered $\Delta$-complex for
$|K|/\tau$, even when distinct orbits have the same quotient
vertex set. Choose a lift of each quotient vertex. Sheet changes
along edges define a cocycle $c$ representing $w_1$.
Compute $c^{\smile k}$ for $k\leq D$ using the
Alexander--Whitney formula and test whether each is a coboundary
by Gaussian elimination over $\mathbb F_2$.
The formula computes the cup product on ordered
$\Delta$-complexes~\cite[Section~3.2]{Hatcher}.

For the second algorithm, compute cohomology and cup products
through degree $D$, and close bases in degrees $1,\ldots,r$
together with the unit under products and linear spans.
Products in $A_r$ of degree greater than $D$ vanish, so this
gives a graded basis and multiplication table for all of $A_r$.
Its cup-length is the largest $\ell$ for which
$(A_r^{>0})^\ell\ne0$, with value zero if $A_r^{>0}=0$;
compute these powers by successively multiplying bases and
taking spans. All operations take time polynomial in the number
of enumerated faces over a fixed finite field. Hence the running
time is $L^{O_r(D)}=\exp(O_r((\log L)^3))$, with an absolute
constant in the first case.
\end{proof}

\section{Combinatorial and geometric applications}
\label{app:sign-complexes}\label{sec:graphs}

We will now apply the topological FLM inequality to obtain a variety of bounds on the index of naturally arising simplicial complexes.

\subsection{Sign complexes and shattering bounds}

Let $A\in\{-1,+1,*\}^{M\times N}$ be a partial sign matrix.
Its sign complex $S(A)$ is generated by the signed row simplices
\(
 \sigma_r=
 \{j^+:r_j=+1\}\cup\{j^-:r_j=-1\}
\)
and their antipodes; see~\cite[Section~3.1]{FHV}.
Exchanging signs gives a free involution. There are at most $2N$
vertices and $2M$ facets. Every finite free simplicial complex arises as such
a sign complex~\cite[Lemma~15]{FHV}.

A set of columns $J \subseteq [N]$ is \emph{shattered} by $A$ if every $\{-1,+1\}$
sign pattern on it occurs as a restriction of a row. Put $d_\infty(A)=\operatorname{VC}(A)$ for the VC dimension of $A$, the largest size of a set of columns shattered by $A$. Write $A^\pm$ for the matrix obtained by adjoining the negatives
of all rows. Let $d_2(A)$ be the largest size
of a column set $J$ for which every nonempty subset $J'\subset J$  of size at most two
 is shattered by $A^\pm$.  

\begin{theorem}[Shattering bounds]
\label{thm:shattering-index}
\label{lem:flag-index}\label{cor:sign}\label{cor:partial}
Let $A\in\{-1,+1,*\}^{M\times N}$, where $M,N\geq1$, and put
$d_\infty=d_\infty(A)$ and $d_2=d_2(A)$. Then
\begin{equation}\label{eq:shattering-basic-pairwise}
 \ind|S(A)|\leq
 \min\{4\log(2M)\log(2N),\,4d_2\log^2(2N)\}.
\end{equation}
For $N\geq2$, there is also the bound
\begin{equation}\label{eq:shattering-full}
 \ind|S(A)|\leq
 C\min\{d_\infty^2\log^2N,\,d_\infty\log^3N\},
\end{equation}
where $C>0$ is absolute.

In particular, if $K$ is a nonempty free flag complex on $v$
vertices, and $d_2$ is the largest number of antipodal vertex pairs
inducing a crosspolytope boundary, then
\begin{equation}\label{eq:shattering-flag}
 \ind|K|\leq4d_2\log^2v.
\end{equation}
\end{theorem}

\begin{proof}
Put $K=S(A)$. If $K$ is empty, the bounds are immediate,
so assume that $K$ is nonempty. Suppose $L\supseteq K$ is a free simplicial complex on the
same signed vertex labels, with at most $f$ facets.
Theorem~\ref{thm:main} gives
\begin{equation}\label{eq:shattering-enlargement}
 \ind|K|\leq\ind|L|\leq4\log(2N)\log f.
\end{equation}

\paragraph{Row count and disambiguation.}
Taking $L=K$ and $f\leq2M$ proves the first bound in
\eqref{eq:shattering-basic-pairwise}.
A total matrix $\bar A$ \emph{disambiguates} $A$ if each row of $A$
is extended by a row of $\bar A$, agreeing at every specified entry.
Such a matrix gives a free enlargement
$S(A)\subseteq S(\bar A)$.
If $d_\infty=0$, all specified entries in any column agree, so a single
total row extends every row. The complex lies in two disjoint
opposite simplices, and $\ind|K|\leq0$, proving
\eqref{eq:shattering-full} in this case.
For $N\geq2$ and $d_\infty\geq1$, the quasipolynomial
Sauer--Shelah--Perles lemma~\cite[Theorem~12]{AHHM} gives a
disambiguation with at most $N^{c d_\infty\log N}$ rows.
Using $f\leq2N^{c d_\infty\log N}$ in
\eqref{eq:shattering-enlargement} gives
\[
 \ind|K|\leq
 4\log(2N)(\log2+c d_\infty\log^2N)=O(d_\infty\log^3N).
\]

\paragraph{Pairwise shattering.}
Let $\widehat K$ be the flag completion of $K$, obtained by filling
all cliques in its $1$-skeleton. There is no antipodal edge, so
$\widehat K$ is still free. Its largest induced antipodal
crosspolytope boundary uses exactly $d_2=d_2(A)$ vertex pairs:
the edges between two different pairs encode the four strict
sign patterns on their two columns.

Write $V=V(\widehat K)$, let $\tau$ exchange signs, and let $N[x]$
be the closed neighborhood of $x$ in the $1$-skeleton.
For every nonempty $C\subseteq V$ with $|C|\leq d_2$, put
\[
 W_C=\bigcap_{c\in C}N[c],
 \qquad F_C=W_C\setminus\tau W_C.
\]
Let $L$ be generated by the sets $F_C$. Since
$F_{\tau C}=\tau F_C$ and $F_C\cap\tau F_C=\varnothing$,
this is a free simplicial complex.

We show $\widehat K\subseteq L$. For a nonempty simplex
$\sigma\in\widehat K$, choose an inclusion-maximal nonempty
$C\subseteq\sigma$ such that $C\cup\tau C$ induces a
crosspolytope boundary. Then $|C|\leq d_2$ and
$\sigma\subseteq W_C$. If some $x\in\sigma$ also belongs to
$\tau W_C$, both $x$ and $\tau x$ are adjacent to every vertex
of $C$ and, by equivariance, to every vertex of $\tau C$.
Here $x\notin C\cup\tau C$: membership in $C$ would force an
antipodal edge, and membership in $\tau C$ would put an antipodal
pair in $\sigma$. Since $\widehat K$ is flag, the set
$C\cup\tau C\cup\{x,\tau x\}$ induces a larger crosspolytope
boundary, contradicting maximality. Thus
$\sigma\subseteq W_C\setminus\tau W_C=F_C$.

Since there are at most
\[
 \sum_{j=1}^{d_2}\binom{|V|}{j}\leq |V|^{d_2}\leq(2N)^{d_2}
\]
choices of $C$, the complex $L$ has at most $(2N)^{d_2}$
facets.
Applying \eqref{eq:shattering-enlargement} proves
$\ind|K|\leq4d_2\log^2(2N)$.
For a free flag complex $K$ on $v$ vertices, take $A$ to
be its signed facet--vertex incidence matrix, with all
facets as rows and one column for each antipodal vertex
pair. Since $S(A)=K$ is already flag and $2N=v$, the
preceding argument gives \eqref{eq:shattering-flag}.

\paragraph{Quadratic dependence on full shattering.}
It remains to prove the $O(d_\infty^2\log^2N)$ bound for $N\geq2$ and
$d_\infty\geq1$. 
Put $K=S(A)$, $v=2N$, and $k=8d_\infty$.
Realize $K$ in the probability simplex $\Delta$ on the signed column
labels, with the involution exchanging the two signs. For independent
samples $X_1,\ldots,X_k$ with law $p\in\Delta$, set
\[
 E(p)=\Pr[\{X_1,\ldots,X_k\}\notin K].
\]
This is an invariant degree-$k$ polynomial with symmetric $0$--$1$
coefficients, and it vanishes on $|K|$.

At an invariant distribution $p$, conditional on the underlying column
labels, the sampled signs are independent and uniform. If a column label repeats, choose two of its occurrences.
Their signs are opposite with probability $1/2$, in which
case $\{X_1,\ldots,X_k\}$ contains an antipodal pair and
is not a face of $K$. If all $k$ column labels are
distinct, the restrictions of rows specified on all these columns
form a total class of VC dimension at most $d_\infty$. By the
Sauer--Shelah--Perles lemma~\cite{Sauer}, at most
$2\sum_{j=0}^{d_\infty}\binom{k}{j}$ sign patterns, including antipodes, are
faces. Their conditional probability is at most
\[
 2^{1-k}\sum_{j=0}^{d_\infty}\binom{k}{j}
 \leq 2(e/32)^{d_\infty}<\frac12.
\]
Consequently, $E(p)\geq1/2$ whenever $p$ is invariant.
Thus the involution acts freely on $Y=\{E<1/2\}$.
Since $|K|\subseteq E^{-1}(0)$,
Lemma~\ref{lem:sampled-energy-compression}, applied with
$G=\mathbb Z/2$, $v=2N$, $k=8d_\infty$, and
$\varepsilon=1/2$, gives
\[
 \ind|S(A)|
 \leq 8k(k-1)\log^2(2N)
 \leq 512d_\infty^2\log^2(2N).
\]

Together with the disambiguation estimate, this proves
\eqref{eq:shattering-full}.
\end{proof}

\subsection{Total matrices}

\begin{corollary}\label{cor:total}
Let $A$ be a total sign matrix with $N\geq 2$ columns and
$d_\infty(A)=d_\infty\geq 1$. Then
\[
 \operatorname{ind}|S(A)|
 \leq 4d_\infty\log(2eN/d_\infty)\log(2N)
 =O(d_\infty\log^2 N),
\]
irrespective of the number of rows. In addition, for each fixed $d_\infty$,
\begin{equation}\label{eq:total-vc-refinement}
 \operatorname{ind}|S(A)|
 =O_{d_\infty}\!\left((\log(2N))^{2d_\infty/(d_\infty+1)}\right).
\end{equation}
\end{corollary}

\begin{proof}
Rows with equal sign patterns give equal simplices, so we may
assume that the rows are distinct. Write $M$ for their number.
By the Sauer--Shelah--Perles lemma \cite{Sauer},
\[
 M\leq\sum_{i=0}^{d_\infty}\binom{N}{i}
 \leq(eN/d_\infty)^{d_\infty}.
\]
Consequently,
\(
 \log(2M)\leq d_\infty\log(2eN/d_\infty).
\)
Theorem~\ref{cor:sign} gives the first bound.

We prove the refinement using the singular-value form of
\eqref{eq:entropy-count}. Identify $|S(A)|$ with $R_1(A)$
via the standard realization of the signed vertices as
$\pm e_j$ on the $\ell_1$-sphere.
At any point $p\in\Delta^\circ$, let $q$ be the Gibbs
distribution from the proof of Theorem~\ref{thm:matrix}, and set
\[
 Q=C^{\mathsf T}C
   =\sum_{i=1}^{2M}q_i r_i r_i^{\mathsf T},
 \qquad
 r_i=(B_{ij}\sqrt{p_j})_{j=1}^{2N}.
\]
Define a probability distribution on the original columns by
\(
 \mu_j=p_j+p_{N+j},\ j=1,\ldots,N.
\)
By Haussler's packing theorem \cite{Haussler1995},
for every $0<\eta\leq1$ there are
$O_{d_\infty}(\eta^{-d_\infty})$ representative rows of $A$ such that
each row disagrees with one of them on columns of total
$\mu$-mass at most $\eta$.
Including the negatives of the corresponding rows of $B$
at most doubles the number of representatives.

Let $\ell'$ be a representative for row $\ell$ such that
$\mu\{j:A_{\ell j}\ne A_{\ell'j}\}\leq\eta$.
Let $i,i'$ be the corresponding row indices in the same
row block of $B$. Then
\[
 \begin{aligned}
 \|r_i-r_{i'}\|_2^2
 &=\sum_{j=1}^{2N}p_j(B_{ij}-B_{i'j})^2
 =4\mu\{j:A_{\ell j}\ne A_{\ell'j}\}
 \leq4\eta.
 \end{aligned}
\]
Let $V$ be the span of these weighted representatives,
and let $P_V$ denote the orthogonal projection onto $V$.
Then
\(
 \dim V=O_{d_\infty}(\eta^{-d_\infty})
\)
and
\[
 \begin{aligned}
 \operatorname{tr}((I-P_V)Q)
 &=\sum_iq_i\|(I-P_V)r_i\|_2^2
 \leq4\eta.
 \end{aligned}
\]
The eigenvalue variational principle therefore gives
\(
 \#\{j:\lambda_j(Q)\geq s\}
 \leq\dim V+\frac{4\eta}{s}.
\)

Let $U$ be the span of the eigenvectors of $Q$ with
eigenvalue at least $s$, and put $t=\dim(U\cap V^\perp)$.
Then $t\geq\dim U-\dim V$. If $e_1,\ldots,e_t$ is an
orthonormal basis of $U\cap V^\perp$, the positive
semidefiniteness of $Q$ gives
\[
st\leq\sum_{a=1}^t e_a^{\mathsf T}Qe_a
\leq\operatorname{tr}((I-P_V)Q)\leq4\eta.
\]
Thus $\dim U\leq\dim V+4\eta/s$, as claimed.
Choosing $\eta=s^{1/(d_\infty+1)}$ gives
\[
 \#\{j:\lambda_j(Q)\geq s\}
 =O_{d_\infty}\!\left(s^{-d_\infty/(d_\infty+1)}\right),
 \qquad 0<s\leq1.
\]

By \eqref{eq:entropy-count}, the number of nonnegative
eigenvalues of $\Hess_{\Delta^\circ}\mathcal F(p)$ is at most
the number of eigenvalues of $Q$ that are at least $s=ab$.
For $\rho=1$, the parameters in the proof of
Theorem~\ref{thm:matrix} are
\(
 a=\frac1{2\log(2N)},\
 b=\frac1{2\log(2M)},
\)
so that
\(
 s=ab=\frac1{4\log(2M)\log(2N)}.
\)
The preceding eigenvalue estimate, followed by the same
application of Lemma~\ref{lem:morse}, gives
\[
 \operatorname{ind}|S(A)|
 =O_{d_\infty}\!\left(
   (\log(2M)\log(2N))^{d_\infty/(d_\infty+1)}
 \right).
\]
Finally,
\(
 \log(2M)
 \leq d_\infty\log(2eN/d_\infty)
 =O_{d_\infty}(\log(2N)),
\)
and hence
\[
 \operatorname{ind}|S(A)|
 =O_{d_\infty}\!\left(
   (\log(2N))^{2d_\infty/(d_\infty+1)}
 \right),
\]
as claimed.
\end{proof}

\subsection{Comparison with sign-rank and questions on VC dimension}

For comparison, Alon, Moran, and Yehudayoff \cite{AMY}
construct total $N\times N$ matrices of VC dimension $2$
whose sign-rank is of order $\sqrt N$ up to logarithmic factors,
whereas \eqref{eq:total-vc-refinement} bounds their index by
$O(\log^{4/3}N)$.

Corollary~\ref{cor:total} gives a bound toward Question~9 of
\cite{FHV}, which asks whether
$\operatorname{coind}|S(A)|$ is bounded by a function of
$d_\infty(A)$ alone for total matrices.

Chornomaz, Moran, and Waknine~\cite{chornomaz2025spherical}
previously posed an equivalent question for spherical dimension.
Let $H$ be the row class of $A$, put
$H^{\pm}=H\cup(-H)$, and let $\Delta(H)$ be the union of its row
simplices before adjoining antipodes. Spherical dimension
$\operatorname{sd}(H)$ is the coindex of
$\Delta(H)\cap(-\Delta(H))$, whereas $|S(A)|$ is
$\Delta(H)\cup(-\Delta(H))$. Hence
\[
 \operatorname{coind}|S(A)|=\operatorname{sd}(H^{\pm}),
 \qquad
 d_\infty(H^{\pm})\leq2d_\infty(H)+1,
\]
where the VC bound follows from the Sauer--Shelah--Perles lemma.
Conversely,
$\Delta(H)\cap(-\Delta(H))\subseteq|S(A)|$.
Thus a bound by a function of VC dimension for either invariant would
give one for the other.

Since $\operatorname{coind}\leq\operatorname{ind}$,
Corollary~\ref{cor:total} also gives the uniform bound
\(
 \operatorname{coind}|S(A)|=O(d_\infty\log^2 N)
\)
and, for each fixed $d_\infty$, the stronger bound
\(
 \operatorname{coind}|S(A)|
 =O_{d_\infty}\!\left((\log(2N))^{2d_\infty/(d_\infty+1)}\right).
\)
Removing the dependence on $N$ remains open.

Blondal, Hatami, Hatami, Lalov, and Tretiak ask the stronger
question of whether the index itself is bounded by a function of
VC dimension \cite[Problem~1.18]{BHHHT} in the case of total matrices.
They also conjecture a product-of-logarithms bound for the
list-replicability number $\operatorname{LR}(A)$
\cite[Conjecture~1.17]{BHHHT}. Their inequality
\(
 \operatorname{ind}|S(A)|\leq2\operatorname{LR}(A)-1
\)
\cite[Theorem~1.1]{BHHHT} shows that this conjecture would imply
the total-matrix case of Theorem~\ref{cor:sign}.

\subsection{Graphs and antipodal configurations}
\label{subsec:graph-applications}

\paragraph{Box complexes and clique number.}
For a finite simple graph $G$ on $V$, let $A^+$ and $A^-$ denote
the two signed copies of a set $A\subseteq V$, and put
\[
 B_0(G)=\{A^+\cup B^-:A,B\subseteq V,\ A\cap B=\varnothing,
                  \ ab\in E(G)\text{ for all }a\in A,\ b\in B\}.
\]
Note that either $A$ or $B$ can be empty. Exchanging signs gives a free involution,
and the standard box-complex bound is
\[
 \chi(G)\geq\ind|B_0(G)|+1\geq\coind|B_0(G)|+1 \geq \omega(G);
\]
see~\cite{ST}. For an integer $t\geq1$, following
\cite[Definition~3]{ST}, $G$ is \emph{topologically $t$-chromatic}
if $\coind|B_0(G)|\geq t-1$.

\begin{corollary}[Clique number and box-complex index]\label{cor:graph}
Every finite simple graph $G$ on $n\geq1$ vertices satisfies
\begin{equation}\label{eq:graphindex}
 \ind|B_0(G)|\leq4\omega(G)\log^2(2n).
\end{equation}
If $\omega(G)\leq r$, where $r\geq1$, and $G$ is topologically
$t$-chromatic, then
\begin{equation}\label{eq:graphsize}
 n\geq\frac12\exp\!\left(\frac12\sqrt{\frac{t-1}{r}}\right).
\end{equation}
\end{corollary}

\begin{proof}
The complex $B_0(G)$ is flag, since its simplices are
precisely the vertex sets containing no antipodal pair
and satisfying all required cross-edge conditions.  Antipodal pairs indexed by
$W\subseteq V$ induce a crosspolytope boundary exactly when $W$ is a
clique in $G$. Thus the parameter in Theorem~\ref{lem:flag-index} is
$d_2=\omega(G)$, giving \eqref{eq:graphindex}. Finally,
$t-1\leq\ind|B_0(G)|\leq4r\log^2(2n)$ gives
\eqref{eq:graphsize}.
\end{proof}

\paragraph{Antipodal vertices and touching translates.}
Two vertices of a polytope $P$ are \emph{antipodal with respect
to $P$} if they lie on distinct parallel supporting hyperplanes
of $P$. Let $a(P)$ be the maximum size of a subset of vertices
that are pairwise antipodal in this sense.

\begin{corollary}[Antipodal vertices of a polytope]
\label{cor:antipodal-polytope}
Every $d$-dimensional convex polytope $P$, $d\geq1$, with $n$
vertices satisfies
\begin{equation}\label{eq:polytope-antipodality}
 a(P)\geq\frac{d}{4\log^2(2n)}.
\end{equation}
Consequently, $P$ admits at least $d/(4\log^2(2n))$ translates
with pairwise disjoint interiors, all sharing a common vertex.
No symmetry or simpliciality assumption is required.
\end{corollary}

\begin{proof}
Identify the affine hull of $P$ with $\mathbb R^d$, write its
vertices as $p_1,\ldots,p_n$, and let $G_P$ be their antipodality
graph. Thus $\omega(G_P)=a(P)$. Form the centrally symmetric
$(d+1)$-dimensional polytope
\[
 Q=\operatorname{conv}\{q_i^+,q_i^-:1\leq i\leq n\},
 \qquad q_i^+=(p_i,1),\quad q_i^-=(-p_i,-1).
\]
All the displayed points are vertices: the sections of $Q$ at
heights $1$ and $-1$ are its two opposite facets, affinely
isomorphic to $P$ and $-P$.

We show that the signed vertex set of every proper face of $Q$
is a simplex of $B_0(G_P)$. Expose such a face by a nonzero
functional $(u,t)\in\mathbb R^d\times\mathbb R$, and put
\(
 h=\max_{z\in Q}\langle(u,t),z\rangle>0.
\)
If this face contains both $q_i^+$ and $q_j^-$, then
\(
 \langle u,p_i\rangle=h-t,
\) and \(
 \langle u,p_j\rangle=-h-t,
\)
while every $p\in P$ satisfies
\[
 -h-t\leq\langle u,p\rangle\leq h-t.
\]
These are distinct parallel supporting hyperplanes of $P$,
so $p_i p_j\in E(G_P)$. Thus the signed vertex set of the face satisfies the
cross-edge condition in the definition of $B_0(G_P)$.
If the face contains only vertices $q_i^+$ or only
vertices $q_j^-$, this condition is vacuous.

To obtain a continuous map, use the barycentric subdivision of
$\partial Q$. Send its vertex corresponding to a nonempty proper
face $F$ to the barycenter of the signed vertex set of $F$, and
extend linearly on chains of faces. On each chain all images lie
in the simplex corresponding to its largest face, so the map is
well defined. The construction commutes with negation. Since
$\partial Q\cong_{\Z/2}S^d$, it gives an equivariant map
\(
 S^d\longrightarrow |B_0(G_P)|.
\)
Consequently, Corollary~\ref{cor:graph} gives
\[
 d\leq\ind|B_0(G_P)|
 \leq4a(P)\log^2(2n),
\]
which proves \eqref{eq:polytope-antipodality}.

Finally, choose pairwise antipodal vertices $p_1,\ldots,p_k$.
All translates $P-p_i$ have the origin as a vertex. By antipodality, for each $i\ne j$ there is a nonzero linear
functional $\ell$ that is maximized on $P$ at $p_i$ and
minimized at $p_j$. Then  
\[
 \operatorname{int}(P-p_i)\subseteq\{x:\ell(x)<0\},
 \qquad
 \operatorname{int}(P-p_j)\subseteq\{x:\ell(x)>0\}.
\]
Their interiors are therefore disjoint.
\end{proof}

\paragraph{Antipodal contacts in sphere covers.}
\begin{corollary}[Pairwise antipodal contacts]
\label{cor:sphere-cover-contacts}
Suppose $S^d=C_1\cup\cdots\cup C_N$, where $d\geq1$, each $C_i$
is closed, and $C_i\cap(-C_i)=\varnothing$. Then some
$I\subseteq[N]$ satisfies
\[
 |I|\geq\frac{d}{4\log^2(2N)},
 \qquad
 C_i\cap(-C_j)\ne\varnothing
 \quad\text{for all distinct }i,j\in I.
\]
\end{corollary}

\begin{proof}
Discarding empty members only strengthens the bound, so assume
every $C_i$ is nonempty. Let $G$ be the graph on $[N]$
with $ij\in E(G)$ precisely when $i\ne j$ and
$C_i\cap(-C_j)\ne\varnothing$. This relation is symmetric.
Since the sets are compact and the family is finite, there
exist open neighborhoods $U_i\supseteq C_i$ such that
\[
 U_i\cap(-U_j)=\varnothing
 \quad\text{whenever }C_i\cap(-C_j)=\varnothing,
\]
including $i=j$. For example, choose a positive radius smaller than half the
distance between every disjoint pair $C_i,-C_j$, and take
$U_i$ to be the open neighborhood of $C_i$ of that radius.

Choose a continuous partition of unity $\varphi_1,\ldots,\varphi_N$
subordinate to the open cover $(U_i)$. In the barycentric
realization on the signed vertices, define
\[
 \Phi(x)=\frac12\sum_{i=1}^N\varphi_i(x)i^+
        +\frac12\sum_{i=1}^N\varphi_i(-x)i^-.
\]
If both $i^+$ and $j^-$ have positive coefficients, then
$x\in U_i\cap(-U_j)$, so $i\ne j$ and $ij\in E(G)$.
Thus $\Phi(x)\in|B_0(G)|$. The map is continuous and
$\Phi(-x)=\tau\Phi(x)$. Corollary~\ref{cor:graph} gives
\[
 d\leq\ind|B_0(G)|\leq4\omega(G)\log^2(2N).
\]
The vertex set of a maximum clique in $G$ is the required $I$.
\end{proof}
\subsection{Equilateral sets and exact $\ell_1$ subspaces}
\label{app:banach-spaces}

All normed spaces in this subsection are real and finite-dimensional.
For a nonzero normed space~$X$, write
\[
 \lambda(X)=
 \max\{s\geq1:
       \ell_1^s\text{ embeds linearly isometrically into }X\}.
\]
Thus $\lambda(X)\geq s$ precisely when there are vectors
$u_1,\ldots,u_s\in X$ satisfying
\begin{equation}\label{eq:banach-isometric-identity}
 \left\|\sum_{j=1}^s a_ju_j\right\|_X
 =\sum_{j=1}^s|a_j|
 \qquad(a_1,\ldots,a_s\in\mathbb R).
\end{equation}
We use the standard correspondence between shattering by dual
functionals and $\ell_1$ structure; see, for example,
\cite[Lemma~4.4]{MendelsonVershynin2004}.

\paragraph{The facet--vertex matrix.}

Let $P=B_X$ be a polytope, and write its vertex set as
\[
 V(P)=\{\pm v_1,\ldots,\pm v_n\}.
\]
Define a partial sign matrix $A_P$, with one row for every
facet $F$ of $P$ and one column for every antipodal vertex pair, by
\[
 (A_P)_{F,j}=
 \begin{cases}
  +1,&v_j\in F,\\
  -1,&-v_j\in F,\\
  *,&v_j,-v_j\notin F.
 \end{cases}
\]
Since no proper face contains an antipodal pair,
the matrix is well defined. Its row class is closed
under negation.
\begin{lemma}[Facet incidence and isometric $\ell_1$ dimension]
\label{lem:banach-incidence}
Let $X$ be a $d$-dimensional normed space with polytopal
unit ball $P$. Then
\begin{equation}\label{eq:banach-incidence-identities}
 \ind|S(A_P)|=d-1,
 \qquad
 d_\infty(A_P)=\lambda(X).
\end{equation}
Moreover, an isometric copy of $\ell_1^{\lambda(X)}$ can be
generated by vertices of $P$. The parameter $d_2(A_P)$ is the
maximum number of distinct antipodal vertex pairs with the
property that any two selected directions $u_i,u_j$ satisfy
\[
 \|\alpha u_i+\beta u_j\|_X=|\alpha|+|\beta|
 \qquad(\alpha,\beta\in\mathbb R).
\]
\end{lemma}

\begin{proof}
First consider the index. Send the signed vertices of $S(A_P)$
to the corresponding vertices $\pm v_j$ of $P$ and extend
affinely over each simplex. Every generating simplex maps
into its corresponding facet, so this defines an equivariant
map
\(
 |S(A_P)|\longrightarrow\partial P.
\)

Conversely, the cell-by-cell construction in the proof of
Theorem~\ref{thm:main} gives an equivariant map
$\partial P\to|S(A_P)|$. Explicitly, send vertices to their
signed labels and extend over one face in each antipodal pair,
keeping each face inside the simplex on its vertices.
The boundary has already mapped into that simplex, which is
contractible; define the extension on the opposite face by
equivariance. Since $\partial P$ is equivariantly homeomorphic
to $S^{d-1}$, monotonicity gives
$\ind|S(A_P)|=d-1$.

Now suppose that $s$ columns, represented by
$u_1,\ldots,u_s\in V(P)$, are shattered.
For every $\varepsilon\in\{-1,+1\}^s$, there is a facet
containing $\varepsilon_1u_1,\ldots,\varepsilon_su_s$.
Its normalized supporting functional $f_\varepsilon$ satisfies
\[
 \|f_\varepsilon\|_{X^*}=1,
 \qquad
 f_\varepsilon(u_j)=\varepsilon_j
 \quad(1\leq j\leq s).
\]
Given scalars $a_1,\ldots,a_s$, choose
$\varepsilon_j=\operatorname{sgn}(a_j)$, with either choice
when $a_j=0$. Then
\[
 \left\|\sum_j a_ju_j\right\|_X
 \geq f_\varepsilon\left(\sum_j a_ju_j\right)
 =\sum_j|a_j|.
\]
The reverse inequality is the triangle inequality.
Thus $\lambda(X)\geq d_\infty(A_P)$.

For the converse, suppose $u_1,\ldots,u_s$ satisfy
\eqref{eq:banach-isometric-identity}; they need not be vertices.
For each $\varepsilon\in\{-1,+1\}^s$, choose a norm-one
functional $f_\varepsilon$ norming
$\sum_j\varepsilon_ju_j$. Since this vector has norm $s$,
\[
 s=\sum_j\varepsilon_jf_\varepsilon(u_j),
 \qquad
 \varepsilon_jf_\varepsilon(u_j)\leq1,
\]
so $f_\varepsilon(u_j)=\varepsilon_j$ for every $j$.

For each $j$, the set
\[
 Q_j=
 \bigcap_{\varepsilon\in\{-1,+1\}^s}
 \{x\in P:\varepsilon_jf_\varepsilon(x)=1\}
\]
is a nonempty face of $P$, since it contains $u_j$.
Choose a vertex $w_j\in Q_j$. Then
$f_\varepsilon(w_j)=\varepsilon_j$ for every $\varepsilon$.
These identities also show that the $w_j$ belong to distinct
antipodal vertex pairs.

For each $\varepsilon$, the vertices
$\varepsilon_1w_1,\ldots,\varepsilon_sw_s$ lie in the exposed
face $\{x\in P:f_\varepsilon(x)=1\}$, hence in a facet.
Write $w_j=\delta_jv_{i_j}$, where $\delta_j\in\{-1,+1\}$.
These facets realize the patterns
$(\varepsilon_j\delta_j)_{j=1}^s$ on columns
$i_1,\ldots,i_s$. As $\varepsilon$ ranges over
$\{-1,+1\}^s$, every sign pattern occurs, so
$d_\infty(A_P)\geq s$.
Taking $s=\lambda(X)$ gives $d_\infty(A_P)=\lambda(X)$,
and the vertices $w_1,\ldots,w_s$ generate an isometric
copy of $\ell_1^s$.

Finally, two vertex columns are shattered precisely when their
representatives satisfy the displayed pairwise norm identity.
Shattering implies the pairwise norm identity by the
preceding argument with $s=2$. Conversely, if the norm
identity holds, norming functionals at $u_i+u_j$ and
$u_i-u_j$ realize the sign patterns $(+,+)$ and $(+,-)$,
and their negatives realize the remaining two patterns. Each corresponding exposed face is
contained in a facet. This proves the interpretation of $d_2(A_P)$.
\end{proof}

\begin{theorem}[Equilateral sets and exact $\ell_1$ subspaces]
\label{cor:equilateral-polyhedral}\label{thm:banach-l1}
Let $X$ be a $d$-dimensional real normed space, $d\geq1$, whose
unit ball has at most $2N$ vertices. There exist distinct antipodal
vertex pairs $\{\pm u_1\},\ldots,\{\pm u_{d_2}\}$, with
\begin{equation}\label{eq:equilateral-pairs}
 d_2\geq\frac{d-1}{4\log^2(2N)},
\end{equation}
such that
\begin{equation}\label{eq:pairwise-ell-one}
 \|\alpha u_i+\beta u_j\|_X=|\alpha|+|\beta|
 \qquad(i\ne j,\ \alpha,\beta\in\mathbb R).
\end{equation}
Consequently, $\{\pm u_1,\ldots,\pm u_{d_2}\}$ is equilateral with
common distance $2$ and has at least $(d-1)/(2\log^2(2N))$ elements.

There also exist vertices $w_1,\ldots,w_{d_\infty}$ satisfying
\[
 \left\|\sum_{j=1}^{d_\infty} a_jw_j\right\|_X=\sum_{j=1}^{d_\infty}|a_j|
 \qquad(a\in\mathbb R^{d_\infty}),
\]
where, for an absolute constant $c>0$,
\begin{equation}\label{eq:banach-l1-bound}
 d_\infty\geq c\max\left\{
 \frac{\sqrt{d-1}}{\log(2N)},\,
 \frac{d-1}{\log^3(2N)}
 \right\}.
\end{equation}
The two collections need not be the same. Both conclusions
apply to $d$-dimensional quotients of $\ell_1^N$ with their
quotient norms.
\end{theorem}

\begin{proof}
The case $d=1$ is immediate. Otherwise let $2n\leq2N$ be the
number of vertices of $B_X$, put $A=A_{B_X}$, and write
$d_2=d_2(A)$ and $d_\infty=d_\infty(A)=\lambda(X)$.
Lemma~\ref{lem:banach-incidence} and
Theorem~\ref{thm:shattering-index} give, after enlarging the
absolute constant if necessary,
\[
 d-1\leq\min\left\{
 4d_2\log^2(2N),\,
 Cd_\infty^2\log^2(2N),\,
 Cd_\infty\log^3(2N)
 \right\}.
\]
Rearranging gives \eqref{eq:equilateral-pairs} and
\eqref{eq:banach-l1-bound}. By Lemma~\ref{lem:banach-incidence},
the corresponding vectors can be chosen among the vertices
of $B_X$ and satisfy the stated norm identities.
Equation~\eqref{eq:pairwise-ell-one} gives distance $2$
between selected signed vertices from different
antipodal pairs. Within each pair,
$\|u_i-(-u_i)\|_X=2\|u_i\|_X=2$, since $u_i$ is a
vertex of the unit ball. The assertion about quotients of
$\ell_1^N$ follows from the unit-ball description below.

\paragraph{Quotients of $\ell_1^N$.}
If $Q:\ell_1^N\to X$ is a surjective linear map and $X$ carries
the quotient norm, then
\[
 B_X=Q(B_{\ell_1^N})
 =\operatorname{conv}\{\pm Qe_1,\ldots,\pm Qe_N\},
\]
so $B_X$ has at most $2N$ vertices.
Conversely, if
$B_X=\operatorname{conv}\{\pm v_1,\ldots,\pm v_n\}$,
the map
\[
 Q:\ell_1^n\longrightarrow X,
 \qquad Q(a)=\sum_{j=1}^n a_jv_j,
\]
induces precisely the given norm as its quotient norm.
If $n\leq N$, the same holds for a map from $\ell_1^N$ by
sending the additional coordinate vectors to zero.
Thus Theorem~\ref{thm:banach-l1} applies to all
$d$-dimensional quotients of $\ell_1^N$.
For each fixed $A\geq1$, the condition $N\leq d^A$ gives
an isometric $\ell_1^{d_\infty}$ with
\[
 d_\infty=\Omega_A(d/\log^3d).
\]
\end{proof}
  \subsection{Vietoris--Rips complexes of spheres}\label{sec:rips}

Throughout this section, $S^d$ carries its geodesic metric
\(
 d_{S^d}(x,y)=\arccos\langle x,y\rangle,
\)
and $\RP^d$ carries the quotient metric
\[
 d_{\RP^d}([x],[y])
 =\min\{d_{S^d}(x,y),d_{S^d}(x,-y)\}
 =\arccos|\langle x,y\rangle|.
\]
For $V\subseteq S^d$, the Vietoris--Rips complex
$\operatorname{VR}(V,r)$ has vertex set $V$ and consists
of its finite subsets of diameter at most $r$.
We use the usual simplicial topology on its geometric
realization. If $V=-V$ and $r<\pi$, no simplex contains
an antipodal pair, so the antipodal involution is free.

Recall the projective packing and covering numbers
\[
 P_d(\delta)
 =
 \max\left\{
 |D|:\ D\subseteq\RP^d,\
 d_{\RP^d}(x,y)\geq\delta
 \text{ for all distinct }x,y\in D
 \right\},
\]
and
\[
 C_d(\epsilon)
 =
 \min\left\{
 |D|:\ D\subseteq\RP^d,\
 \RP^d=\bigcup_{x\in D}B(x,\epsilon)
 \right\}.
\]

Adams-Bush-Frick~\cite[Theorem~3]{ABFProjectiveCodes} obtained lower bounds on the coindex of rips complexes in terms of packing numbers. Based on this idea, Adams-Bush-Virk~\cite{ABV} obtain a covering-based upper bound for the index of rips complex. Together, these give for $\pi/2\le r<\pi$, 
\[
P_d(\pi-r)-1
\leq \operatorname{coind}|\operatorname{VR}(S^d,r)|
\leq \operatorname{ind}|\operatorname{VR}(S^d,r)|
\leq C_d\left(\frac{\pi-r}{2}\right)-1.
\]
It is natural to wonder how far apart these lower and upper bounds are. It turns out that $P(\pi-r)$ is often exponentially smaller than $C_d\left(\frac{\pi-r}{2}\right)$. In Corollary~\ref{cor:rips} below, we propose an upper bound that depends on both the packing and covering numbers. We will then show that this bound is substantially stronger is certain cases.

\begin{corollary}[Projective packing bounds]\label{cor:rips}
Let $d\geq1$ and $\pi/2\leq r<\pi$.
For every finite antipodal set $V\subseteq S^d$
consisting of $N\geq1$ antipodal pairs,
\begin{equation}\label{eq:rips-finite}
 \ind|\operatorname{VR}(V,r)|
 \leq4 P_d(\pi-r)\log^2(2N).
\end{equation}
Moreover, for every $0<2\epsilon<\pi-r$,
\begin{equation}\label{eq:rips-sphere}
 \ind|\operatorname{VR}(S^d,r)|
 \leq4 P_d(\pi-r-2\epsilon)
       \log^2\!\bigl(2C_d(\epsilon)\bigr).
\end{equation}

\end{corollary}

\begin{proof}[Proof of Corollary~\ref{cor:rips}]
Write $V=\{\pm v_1,\ldots,\pm v_N\}$. Antipodal pairs indexed by
$J\subseteq[N]$ induce a crosspolytope boundary precisely when all
distances between distinct pairs, with either choice of signs, are
at most $r$. Since $\pi/2\leq r<\pi$, this is equivalent to the condition
\(
 |\langle v_i,v_j\rangle|\leq-\cos r
\) for all $i\ne j$,
or equivalently \(
 d_{\RP^d}([v_i],[v_j])\geq\pi-r.
\)
Thus the parameter $d_2$ of Theorem~\ref{lem:flag-index} is the maximum
size of a projective $(\pi-r)$-packing supported on the sample, and
$d_2\leq P_d(\pi-r)$. The theorem gives
\[
 \ind|\operatorname{VR}(V,r)|
 \leq4 d_2\log^2(2N)
 \leq4 P_d(\pi-r)\log^2(2N),
\]
proving \eqref{eq:rips-finite}.

Now fix $0<2\epsilon<\pi-r$. Lift a covering of $\RP^d$ by
$C_d(\epsilon)$ open $\epsilon$-balls to an antipodal
$\epsilon$-net $V\subseteq S^d$. Choose a nearest-point map
$\nu:S^d\to V$ equivariantly, by making a choice on one point of
each antipodal pair and prescribing $\nu(-x)=-\nu(x)$. The triangle
inequality gives
\[
 d_{S^d}(\nu(x),\nu(y))\leq d_{S^d}(x,y)+2\epsilon.
\]
Hence $\nu$ induces an equivariant simplicial map
$\operatorname{VR}(S^d,r)\to\operatorname{VR}(V,r+2\epsilon)$.
Although the vertex map need not be continuous in the geodesic
metric, the corresponding map on Vietoris-Rips complexes is simplicial.
Since $r+2\epsilon<\pi$, \eqref{eq:rips-finite} and monotonicity give
\[
 \ind|\operatorname{VR}(S^d,r)|
 \leq4 P_d(\pi-r-2\epsilon)
            \log^2\!\bigl(2C_d(\epsilon)\bigr).\qedhere
\]
\end{proof}
\paragraph{Comparison of the bounds in terms of dimension.}
We will now give several examples of how our bound provides exponential improvements over the bound by a covering numbers. Fix $\pi/2<r<\pi$.  Standard spherical-cap
estimates together with asymptotically efficient coverings give
$C_d(\theta)=\exp\!\left(\left[-\log\sin(\theta)+o(1)\right]d\right)$,
so the direct covering argument gives
\begin{equation}
\label{eq:covering_index_rips_bound}
   \ind|\operatorname{VR}(S^d,r)|
 \leq
 \exp\!\left(
 \left[-\log\sin\!\left(\frac{\pi-r}{2}\right)+o(1)\right]d
 \right).  
\end{equation}

On the other hand, taking $\epsilon=d^{-2}$ in
\eqref{eq:rips-sphere}, the logarithmic covering factor is only
polynomial in $d$, while the Kabatiansky--Levenshtein bound \cite{KabatianskyLevenshtein1978} for
projective codes gives
\begin{equation}
\label{eq:packing_index_rips_upper_bound}
 \ind|\operatorname{VR}(S^d,r)|
 \leq
 \exp\!\left(
 \left[
 \frac{1+\sin(\pi-r)}{2\sin(\pi-r)}
 H\!\left(
 \frac{1-\sin(\pi-r)}{1+\sin(\pi-r)}
 \right)
 +o(1)
 \right]d
 \right),
 \end{equation}
where $H(t)=-t\log t-(1-t)\log(1-t)$.
This exponent is strictly smaller than
$-\log\sin((\pi-r)/2)$, so the packing-based argument gives a
strictly sharper exponential bound for every fixed $\pi/2<r<\pi$.

For example, take $r=2\pi/3$.  Plugging this into equation~(\ref{eq:covering_index_rips_bound}) shows us that the direct covering bound gives
$\ind|\operatorname{VR}(S^d,2\pi/3)|\leq(2+o(1))^d$.
In contrast, equation~(\ref{eq:packing_index_rips_upper_bound}) gives,
$P_d(\pi/3-o(1))\leq \exp((0.279+o(1))d)=(1.321+o(1))^d$.
Thus the packing-based method improves the exponential bound from
approximately $2^d$ to $1.321^d$. 
\paragraph{Bounds near $\mathbf{r=\pi/2}$.}
We now consider the regime where $r-\pi/2=O(d^{-1/2})$. In this range, our packing-based bound is polynomial in $d$, whereas the covering bound is exponential.

Put $a=-\cos r$. Suppose that
$x_1,\ldots,x_m\in S^d$ represent a projective
$(\pi-r)$-packing, so that
\(
 |\langle x_i,x_j\rangle|\leq a.
\)
for $i\ne j$. 
Their Gram matrix
$G=(\langle x_i,x_j\rangle)_{i,j=1}^m$
is positive semidefinite, has rank at most $d+1$,
and satisfies $\operatorname{tr}G=m$. Applying Cauchy--Schwarz
to its nonzero eigenvalues gives
\[
 \frac{m^2}{d+1}
 \leq\operatorname{tr}(G^2)
 =\sum_{i,j=1}^m\langle x_i,x_j\rangle^2
 \leq m+m(m-1)a^2.
\]
If $(d+1)a^2<1$, rearranging gives
\begin{equation}\label{eq:rips-gram-packing}
 P_d(\pi-r)
 \leq\frac{(d+1)(1-a^2)}{1-(d+1)a^2}.
\end{equation}

In particular, if
\(
 \frac{\pi}{2}\leq r\leq
 \frac{\pi}{2}+\frac{1}{2\sqrt{d+1}},
\)
then
\(
 0\leq a=\sin(r-\pi/2)\leq\frac{1}{2\sqrt{d+1}},
\)
and \eqref{eq:rips-gram-packing} gives
\(P_d(\pi-r)\leq\frac{4(d+1)}{3}\).
Substituting into \eqref{eq:rips-finite}, we obtain
\begin{equation}\label{eq:rips-near-hemisphere}
 \operatorname{ind}|\operatorname{VR}(V,r)|
 \leq\frac{16}{3}(d+1)\log^2(2N).
\end{equation}
In particular, for polynomial-size samples $N=d^{O(1)}$, this is $O(d\log^2d)$.

We can obtain a polynomial bound for the full sphere as well. Choose
\(
 \varepsilon=\frac{1}{16\sqrt{d+1}}.
\)
Throughout the same range of $r$, we have
\(
 -\cos(r+2\varepsilon)
 =\sin(r-\pi/2+2\varepsilon)
 \leq\frac{5}{8\sqrt{d+1}}.
\)
Applying \eqref{eq:rips-gram-packing} at scale
$r+2\varepsilon$ gives
\(
 P_d(\pi-r-2\varepsilon)
 \leq\frac{64}{39}(d+1).
\)
Moreover, the standard volumetric covering estimate gives
\(
 C_d(\varepsilon)
 \leq \left(\frac{C}{\varepsilon}\right)^d
 =\exp(O(d\log d))
\)
for an absolute constant $C$. Consequently, the whole-sphere bound gives
\[
 \operatorname{ind}|\operatorname{VR}(S^d,r)|
 \leq \frac{256}{39}(d+1)
       \log^2\bigl(2C_d(\varepsilon)\bigr)
 =O(d^3\log^2d).
\]

However, we will now show that $C_d\left(\frac{\pi-r}{2}\right)$ is exponential in $d$. Indeed, 
since a projective ball of radius $\theta\leq\pi/4$ occupies at most $(\sin\theta)^d\leq 2^{-d/2}$ of the total volume of $\mathbb{RP}^d$, any covering by such balls requires at least $2^{d/2}$ balls. We have $(\pi-r)/2\leq\pi/4$, so then
\(
C_d\left(\frac{\pi-r}{2}\right)\geq 2^{d/2}.
\)

\section*{AI use statement}
The authors developed the topological Figiel--Lindenstrauss--Milman
approach to lower bounds for the vertex numbers of triangulations
of real projective spaces without AI use. LLMs assisted with proof discovery (in particular the use of Lemma~\ref{lem:morse}) and
manuscript drafts. The authors subsequently substantially
simplified the proofs and revised the manuscript.

\section*{Acknowledgements} 
FF thanks Dan Guyer, Amzi Jeffs, Eran Nevo, and many other colleagues
for helpful conversations about triangulations of real projective spaces.  KH thanks the Simons Institute for the Theory of Computing at UC Berkeley where this work was partially conducted.
FF was supported by NSF CAREER Grant DMS 2042428.
EM was supported by the National Science Foundation Graduate Research Fellowship Program under Grant Numbers DGE2140739 and DGE2631988. Any opinions, findings, and conclusions or recommendations expressed in this material are those of the authors and do not necessarily reflect the views of the National Science Foundation.


\begin{thebibliography}{10}

\bibitem{ABFProjectiveCodes}
Henry Adams, Johnathan Bush, and Florian Frick.
\newblock The topology of projective codes and the distribution of zeros of odd
  maps.
\newblock {\em Michigan Mathematical Journal}, 74(4):775--796, 2024.

\bibitem{ABV}
Henry Adams, Johnathan Bush, and {\v{Z}}iga Virk.
\newblock The connectivity of {Vietoris--Rips} complexes of spheres.
\newblock {\em Journal of Applied and Computational Topology}, 9, 2025.
\newblock Article 15.

\bibitem{AAK}
Karim Adiprasito, Sergey Avvakumov, and Roman Karasev.
\newblock A subexponential size triangulation of {$\mathbb{R}P^n$}.
\newblock {\em Combinatorica}, 42(1):1--8, 2022.

\bibitem{Alexandroff1932}
Paul Alexandroff.
\newblock {\em Einfachste Grundbegriffe der Topologie}.
\newblock Springer, Berlin, 1932.

\bibitem{AHHM}
Noga Alon, Steve Hanneke, Ron Holzman, and Shay Moran.
\newblock A theory of {PAC} learnability of partial concept classes.
\newblock In {\em 2021 {IEEE} 62nd Annual Symposium on Foundations of Computer
  Science ({FOCS})}, pages 658--671. IEEE, 2022.
\newblock Full version: arXiv:2107.08444v2.

\bibitem{AMY}
Noga Alon, Shay Moran, and Amir Yehudayoff.
\newblock Sign rank versus {VC} dimension.
\newblock In {\em 29th Annual Conference on Learning Theory}, volume~49 of {\em
  Proceedings of Machine Learning Research}, pages 47--80. PMLR, 2016.

\bibitem{AM}
Pierre Arnoux and Alexis Marin.
\newblock The {K{\"u}hnel} triangulation of the complex projective plane from
  the view point of complex crystallography. {Part II}.
\newblock {\em Memoirs of the Faculty of Science, Kyushu University. Series A,
  Mathematics}, 45(2):167--244, 1991.

\bibitem{AubrunMO}
Guillaume Aubrun.
\newblock Comment on ``unexpected applications of {Dvoretzky}'s theorem''.
\newblock MathOverflow,
  \url{https://mathoverflow.net/questions/46278/unexpected-applications-of-dvoretzkys-theorem},
  2025.
\newblock Comment dated April 6, 2025.

\bibitem{ABHK2024}
Sergey Avvakumov, Alexey Balitskiy, Alfredo Hubard, and Roman Karasev.
\newblock Systolic inequalities for the number of vertices.
\newblock {\em Journal of Topology and Analysis}, 16(6):955--977, 2024.

\bibitem{BagchiDatta2012}
Bhaskar Bagchi and Basudeb Datta.
\newblock A triangulation of {$\mathbb{C}P^3$} as symmetric cube of {$S^2$}.
\newblock {\em Discrete \& Computational Geometry}, 48(2):310--329, 2012.

\bibitem{Balagopalan2017}
Sonia Balagopalan.
\newblock On a vertex-minimal triangulation of {$\mathbb{R}P^4$}.
\newblock {\em The Electronic Journal of Combinatorics}, 24(1):P1.52, 2017.

\bibitem{Barvinok2013}
Alexander Barvinok.
\newblock A bound for the number of vertices of a polytope with applications.
\newblock {\em Combinatorica}, 33(1):1--10, 2013.

\bibitem{BHHHT}
Ari Blondal, Hamed Hatami, Pooya Hatami, Chavdar Lalov, and Sivan Tretiak.
\newblock Sign-rank, index, and list replicability: Connections and
  separations, 2026.
\newblock arXiv:2606.18236.

\bibitem{BorelSerre1953}
Armand Borel and Jean-Pierre Serre.
\newblock Groupes de {Lie} et puissances r{\'e}duites de {Steenrod}.
\newblock {\em American Journal of Mathematics}, 75(3):409--448, 1953.

\bibitem{BrehmKuhnel1987}
Ulrich Brehm and Wolfgang K{\"u}hnel.
\newblock Combinatorial manifolds with few vertices.
\newblock {\em Topology}, 26(4):465--473, 1987.

\bibitem{BrehmKuhnel1992}
Ulrich Brehm and Wolfgang K{\"u}hnel.
\newblock 15-vertex triangulations of an 8-manifold.
\newblock {\em Mathematische Annalen}, 294(1):167--193, 1992.

\bibitem{Brion2005}
Michel Brion.
\newblock Lectures on the geometry of flag varieties.
\newblock In Piotr Pragacz, editor, {\em Topics in Cohomological Studies of
  Algebraic Varieties}, Trends in Mathematics, pages 33--85. Birkh{\"a}user,
  Basel, 2005.

\bibitem{Brown1982}
Kenneth~S. Brown.
\newblock {\em Cohomology of Groups}, volume~87 of {\em Graduate Texts in
  Mathematics}.
\newblock Springer-Verlag, New York, 1982.

\bibitem{Cairns1935}
Stewart~S. Cairns.
\newblock Triangulation of the manifold of class one.
\newblock {\em Bulletin of the American Mathematical Society}, 41(8):549--552,
  1935.

\bibitem{chornomaz2025spherical}
Bogdan Chornomaz, Shay Moran, and Tom Waknine.
\newblock Spherical dimension.
\newblock In {\em Proceedings of Thirty Eighth Conference on Learning Theory},
  volume 291 of {\em Proceedings of Machine Learning Research}, pages
  1259--1313. PMLR, 2025.

\bibitem{Csorba2007}
P{\'e}ter Csorba.
\newblock Homotopy types of box complexes.
\newblock {\em Combinatorica}, 27(6):669--682, 2007.

\bibitem{DaneshpajouhMeunier2025}
Hamid~Reza Daneshpajouh and Fr{\'e}d{\'e}ric Meunier.
\newblock Box complexes: At the crossroad of graph theory and topology.
\newblock {\em Discrete Mathematics}, 348(6), 2025.
\newblock Article 114422.

\bibitem{DattaSpreer2026}
Basudeb Datta and Jonathan Spreer.
\newblock Simplicial cell decompositions of {$\mathbb{CP}^{n}$}.
\newblock {\em Discrete \& Computational Geometry}, 76(2):988--1011, 2026.

\bibitem{DeyRathod2024}
Tamal~K. Dey and Abhishek Rathod.
\newblock Cup product persistence and its efficient computation.
\newblock In {\em 40th International Symposium on Computational Geometry (SoCG
  2024)}, volume 293 of {\em Leibniz International Proceedings in Informatics
  (LIPIcs)}, pages 50:1--50:15. Schloss Dagstuhl -- Leibniz-Zentrum f{\"u}r
  Informatik, 2024.

\bibitem{Dold1983}
Albrecht Dold.
\newblock Simple proofs of some {Borsuk--Ulam} results.
\newblock In {\em Proceedings of the Northwestern Homotopy Theory Conference
  ({Evanston}, {Ill.}, 1982)}, volume~19 of {\em Contemporary Mathematics},
  pages 65--69. American Mathematical Society, Providence, RI, 1983.

\bibitem{DMZ}
Haibao Duan, Wac{\l}aw Marzantowicz, and Xuezhi Zhao.
\newblock On the number of simplices required to triangulate a {Lie} group.
\newblock {\em Topology and its Applications}, 293:107559, 2021.

\bibitem{EilenbergGanea1957}
Samuel Eilenberg and Tudor Ganea.
\newblock On the {Lusternik--Schnirelmann} category of abstract groups.
\newblock {\em Annals of Mathematics}, 65(3):517--518, 1957.

\bibitem{FadellNeuwirth1962}
Edward Fadell and Lee Neuwirth.
\newblock Configuration spaces.
\newblock {\em Mathematica Scandinavica}, 10:111--118, 1962.

\bibitem{Farber2024}
Michael Farber.
\newblock Large simplicial complexes: universality, randomness, and ampleness.
\newblock {\em Journal of Applied and Computational Topology}, 8:1551--1574,
  2024.
\newblock With an appendix by Jonathan Ariel Barmak.

\bibitem{FLM}
T.~Figiel, J.~Lindenstrauss, and V.~D. Milman.
\newblock The dimension of almost spherical sections of convex bodies.
\newblock {\em Acta Mathematica}, 139:53--94, 1977.

\bibitem{FPP}
Peter Frankl, J{\'a}nos Pach, and D{\"o}m{\"o}t{\"o}r P{\'a}lv{\"o}lgyi.
\newblock Exchange properties of finite set-systems.
\newblock {\em SIAM Journal on Discrete Mathematics}, 36(3):2073--2081, 2022.

\bibitem{FHV}
Florian Frick, Kaave Hosseini, and Aliaksei Vasileuski.
\newblock A {$\mathbb{Z}_2$}-topological framework for sign-rank lower bounds,
  2026.
\newblock arXiv:2604.01510.

\bibitem{Gaifullin2024}
Alexander~A. Gaifullin.
\newblock New examples and partial classification of 15-vertex triangulations
  of the quaternionic projective plane.
\newblock {\em Proceedings of the Steklov Institute of Mathematics},
  326(1):52--89, 2024.

\bibitem{GonzalezDiazReal2002}
Roc{\'i}o Gonz{\'a}lez-D{\'i}az and Pedro Real.
\newblock Geometric objects and cohomology operations.
\newblock In {\em Proceedings of the 5th Workshop on Computer Algebra in
  Scientific Computing}, pages 121--130, 2002.

\bibitem{Gorodkov2019}
Denis Gorodkov.
\newblock A 15-vertex triangulation of the quaternionic projective plane.
\newblock {\em Discrete \& Computational Geometry}, 62(2):348--373, 2019.

\bibitem{GMP}
Dejan Govc, Wac{\l}aw Marzantowicz, and Petar Pave{\v{s}}i{\'c}.
\newblock Estimates of covering type and the number of vertices of minimal
  triangulations.
\newblock {\em Discrete \& Computational Geometry}, 63:31--48, 2020.

\bibitem{GMPGrassmannians2020}
Dejan Govc, Wac{\l}aw Marzantowicz, and Petar Pave{\v{s}}i{\'c}.
\newblock How many simplices are needed to triangulate a {Grassmannian}?
\newblock {\em Topological Methods in Nonlinear Analysis}, 56(2):501--518,
  2020.

\bibitem{GMPCategoryWeight}
Dejan Govc, Wac{\l}aw Marzantowicz, and Petar Pave{\v{s}}i{\'c}.
\newblock Estimates of covering type and minimal triangulations based on
  category weight.
\newblock {\em Forum Mathematicum}, 34(4):969--988, 2022.

\bibitem{Gromov1988}
Mikhail Gromov.
\newblock Dimension, non-linear spectra and width.
\newblock In Joram Lindenstrauss and Vitali~D. Milman, editors, {\em Geometric
  Aspects of Functional Analysis}, volume 1317 of {\em Lecture Notes in
  Mathematics}, pages 132--184. Springer, Berlin, Heidelberg, 1988.

\bibitem{GuyerSteinerbergerYang2026}
Dan Guyer, Stefan Steinerberger, and Yirong Yang.
\newblock An efficient triangulation of {$\mathbb{R}P^5$}, 2026.
\newblock arXiv:2603.07808.

\bibitem{Hatcher}
Allen Hatcher.
\newblock {\em Algebraic Topology}.
\newblock Cambridge University Press, 2002.

\bibitem{Haussler1995}
David Haussler.
\newblock Sphere packing numbers for subsets of the {Boolean} {$n$}-cube with
  bounded {Vapnik--Chervonenkis} dimension.
\newblock {\em Journal of Combinatorial Theory, Series A}, 69(2):217--232,
  1995.

\bibitem{He2017}
Chen He.
\newblock Torsions of integral homology and cohomology of real {Grassmannians},
  2017.
\newblock arXiv:1709.05623.

\bibitem{KabatianskyLevenshtein1978}
G.~A. Kabatiansky and V.~I. Levenshtein.
\newblock On bounds for packings on a sphere and in space.
\newblock {\em Problems of Information Transmission}, 14(1):1--17, 1978.
\newblock English translation; Russian original in Problemy Peredachi
  Informatsii 14(1):3--25.

\bibitem{KW}
Max Karoubi and Charles~A. Weibel.
\newblock On the covering type of a space.
\newblock {\em L'Enseignement Math{\'e}matique}, 62(3/4):457--474, 2016.

\bibitem{KleeNovik2012}
Steven Klee and Isabella Novik.
\newblock Centrally symmetric manifolds with few vertices.
\newblock {\em Advances in Mathematics}, 229(1):487--500, 2012.

\bibitem{Kuhnel1987}
Wolfgang K{\"u}hnel.
\newblock Minimal triangulations of {K}ummer varieties.
\newblock {\em Abhandlungen aus dem Mathematischen Seminar der Universit{\"a}t
  Hamburg}, 57:7--20, 1987.

\bibitem{KuhnelBanchoff1983}
Wolfgang K{\"u}hnel and Thomas~F. Banchoff.
\newblock The 9-vertex complex projective plane.
\newblock {\em The Mathematical Intelligencer}, 5(3):11--22, 1983.

\bibitem{KuhnelLassmann1988}
Wolfgang K{\"u}hnel and G.~Lassmann.
\newblock Combinatorial {$d$}-tori with a large symmetry group.
\newblock {\em Discrete \& Computational Geometry}, 3:169--176, 1988.

\bibitem{LeonAlvarezMorales2025}
Porfirio~L. Le{\'o}n~{\'A}lvarez and Israel Morales.
\newblock Asymptotic and cohomological dimension of surface braid groups and
  poly-surface groups, 2025.
\newblock arXiv:2506.10706v2, revised 2026.

\bibitem{Lutz2005}
Frank~H. Lutz.
\newblock Triangulated manifolds with few vertices: Combinatorial manifolds,
  2005.
\newblock arXiv:math/0506372.

\bibitem{Matousek}
Ji{\v{r}}{\'i} Matou{\v{s}}ek.
\newblock {\em Using the {Borsuk--Ulam} Theorem: Lectures on Topological
  Methods in Combinatorics and Geometry}.
\newblock Universitext. Springer-Verlag, Berlin, 2003.

\bibitem{MendelsonVershynin2004}
Shahar Mendelson and Roman Vershynin.
\newblock Remarks on the geometry of coordinate projections in
  {$\mathbb{R}^n$}.
\newblock {\em Israel Journal of Mathematics}, 140:203--220, 2004.

\bibitem{Milnor}
John Milnor.
\newblock {\em {Morse} Theory}.
\newblock Number~51 in Annals of Mathematics Studies. Princeton University
  Press, Princeton, NJ, 1963.

\bibitem{Milo}
Tomer Milo.
\newblock On the {Figiel--Lindenstrauss--Milman} inequality, 2025.
\newblock arXiv:2504.13571.

\bibitem{Moebius1886}
August~Ferdinand M{\"o}bius.
\newblock {\em Gesammelte Werke}, volume~2.
\newblock S. Hirzel, Leipzig, 1886.

\bibitem{Pavesic2019}
Petar Pave{\v{s}}i{\'c}.
\newblock Triangulations with few vertices of manifolds with non-free
  fundamental group.
\newblock {\em Proceedings of the Royal Society of Edinburgh Section A:
  Mathematics}, 149(6):1453--1463, 2019.

\bibitem{Raptis2017}
George Raptis.
\newblock On {Serre} microfibrations and a lemma of {M. Weiss}.
\newblock {\em Glasgow Mathematical Journal}, 59(3):649--657, 2017.

\bibitem{Sauer}
Norbert Sauer.
\newblock On the density of families of sets.
\newblock {\em Journal of Combinatorial Theory, Series A}, 13(1):145--147,
  1972.

\bibitem{ST}
G{\'a}bor Simonyi and G{\'a}bor Tardos.
\newblock Local chromatic number, {Ky Fan}'s theorem, and circular colorings.
\newblock {\em Combinatorica}, 26:587--626, 2006.

\bibitem{VenturelloZheng2021}
Lorenzo Venturello and Hailun Zheng.
\newblock A new family of triangulations of {$\mathbb{R}P^d$}.
\newblock {\em Combinatorica}, 41:127--146, 2021.

\bibitem{Zuckerman2007}
David Zuckerman.
\newblock Linear degree extractors and the inapproximability of {Max Clique}
  and {Chromatic Number}.
\newblock {\em Theory of Computing}, 3(6):103--128, 2007.

\end{thebibliography}
\end{document}